\documentclass[12pt,a4paper]{amsart}
\usepackage{graphicx,amssymb}
\usepackage{amsmath,amssymb,amscd}
\input xy
\xyoption{all}
\usepackage[all]{xy}
\usepackage{tikz}

\usepackage{hyperref}

\newcommand{\lra}{\longrightarrow}

\newcommand{\ra}{\rightarrow}

\newcommand{\CC}{\mathbb C}

\newcommand{\cO}{\mathcal{O}}

\newcommand{\Ext}{\mbox{Ext}}

\newcommand{\Hom}{\mbox{Hom}}

\newcommand{\Cliff}{\operatorname{Cliff}}
\newcommand{\gr}{\mbox{gr}}
\newcommand{\rk}{\mbox{rk}}

\newcommand{\Ker}{\operatorname{Ker}}
\theoremstyle{plain}
\newtheorem{theorem}{Theorem}[section]
\newtheorem{lem}[theorem]{Lemma}
\newtheorem{prop}[theorem]{Proposition}
\newtheorem{cor}[theorem]{Corollary}

\newtheorem{rem}[theorem]{Remark}

\numberwithin{equation}{section}
\begin{document}
\title[BN-theory for genus 6]{Higher rank BN-theory for curves of genus 6}

\author{H. Lange}
\author{P. E. Newstead}

\address{H. Lange\\Department Mathematik\\
              Universit\"at Erlangen-N\"urnberg\\
              Cauerstra\ss e 11\\
              D-$91058$ Erlangen\\
              Germany}
              \email{lange@mi.uni-erlangen.de}
\address{P.E. Newstead\\Department of Mathematical Sciences\\
              University of Liverpool\\
              Peach Street, Liverpool L69 7ZL, UK}
\email{newstead@liv.ac.uk}

\thanks{Both authors are members of the research group VBAC (Vector Bundles on Algebraic Curves). The second author 
would like to thank the Department Mathematik der Universit\"at 
         Erlangen-N\"urnberg for its hospitality}
\keywords{Semistable vector bundle, Brill-Noether theory, genus 6, Clifford index, gonality}
\subjclass[2010]{Primary: 14H60}
\date{\today}

\begin{abstract}
Higher rank Brill-Noether theory for genus 6 is especially interesting as, even in the general case, some unexpected phenomena arise which are absent in lower genus. Moreover, it is the first case for which there exist curves of Clifford dimension greater than 1 (smooth plane quintics). In all cases, we obtain new upper bounds for non-emptiness of Brill-Noether loci and construct many examples which approach these upper bounds more closely than those that are well known. Some of our examples of non-empty Brill-Noether loci have negative Brill-Noether numbers.
\end{abstract}
\maketitle

\section{Introduction}\label{intro}

Let $C$ be a smooth complex projective curve of genus $g$ and let $B(n,d,k)$ denote the {\it Brill-Noether locus} of stable bundles on $C$ of rank 
$n$ and degree $d$ with at least $k$ independent sections (for the formal definition, see Section \ref{back}). This locus has a natural 
structure as a subscheme of the moduli space of stable bundles on $C$ of rank $n$ and degree $d$.

In the case $n=1$, the Brill-Noether loci are classical objects. For $n>1$, 
the study began towards the end of the 1980s and the situation is much less clear; even on a general 
curve, there is a great deal that is not known. The problem is completely 
solved only for $g\le3$ (see \cite{bgn,m1,m2}), although there are strong results for hyperelliptic and bielliptic curves (see \cite{bmno} and \cite{b}) and the authors have obtained many new results for $g=4$ and $g=5$ \cite{ln1,ln2}). 

In this paper, we consider the case where $C$ has genus $6$. This is especially interesting as it is the first case in which it is known that there exist stable bundles with negative Brill-Noether number on a  general curve. One of these bundles arises from another interesting phenomenon. In considering a Brill-Noether locus $B(2,K_C,k)$ of stable bundles of rank 2 with determinant the canonical bundle $K_C$, the definition of Brill-Noether number needs to be changed to reflect the symmetry properties of these bundles. This new Brill-Noether number can be larger than the standard one, leading to the intriguing fact that the ``expected'' dimension of $B(2,K_C,k)$ is larger than the ``expected'' dimension of $B(2,2g-2,k)$. Genus $6$ is not the first case in which this arises, but it is the first for which the canonical Brill-Noether number is non-negative while the corresponding ordinary Brill-Noether number is negative.

As in our work on genus $4$ and $5$, the main results of the paper concern new upper 
bounds on $k$ for the non-emptiness of $B(n,d,k)$ and the corresponding loci $\widetilde{B}(n,d,k)$ for 
semistable bundles. A complete answer is known for $d\le2n$ in any genus (see Proposition \ref{pln2}). For genus $6$, it is therefore sufficient in view of Serre duality to restrict to the range $2n<d\le5n$. There are several different cases to be considered, namely $C$ hyperelliptic, Clifford index $\Cliff(C) =1$ and $\Cliff(C)=2$. For $\Cliff(C)=1$, there is a further divison into trigonal curves and smooth plane quintics, while, for $\Cliff(C)=2$, a special r\^ole is played by bielliptic curves. Smooth plane quintics are the first case in which $C$ does not possess a line bundle with $h^0=2$ which computes the Clifford index; we say in this case that $C$ has Clifford dimension greater than $1$.

In the case $\Cliff(C)=1$, the results are complicated and we refer to subsections \ref{sub4.2} and \ref{sub4.3} and the BN-maps in Section \ref{secBN} for details. For $\Cliff(C)=2$, including the case of a general curve $C$, things are simpler and we have the following theorem.\\

\noindent{\bf Theorem \ref{th3.18}.} \emph{Let $C$ be a curve of genus $6$ with $\Cliff(C) = 2$. If $2n < d\leq 5n$ and 
$\widetilde B(n,d,k) \neq \emptyset$, then one of the following holds,
\begin{enumerate}
\item[(i)] $2n <d \leq \frac{9}{4}n$ and $k \leq n + \frac{1}{5}(d-n)$;
\item[(ii)] $\frac{9}{4}n < d \leq \frac{7}{3}n$ and $k < d-n$;
\item[(iii)] $\frac{7}{3}n < d \leq 3n$ and $k \leq n + \frac{1}{4}(d-n)$;
\item[(iv)] $3n < d \leq 5n$ and $k \leq \frac12d$.
\end{enumerate}
}
The bounds for $3n\le d\le5n$ cannot be improved since they are known to be attainable on any bielliptic curve \cite{b}.

We also produce 
a large number of examples of stable bundles which come close to attaining the upper bounds of our main results.
Many of these are constructed using elementary transformations, the only problem here being to prove stability.
Some of these were already established in \cite{ln1,ln2}, but others are new. For several of our examples, the Brill-Noether number is negative.\\

In Section \ref{back}, we give some background and describe some known results. In Section \ref{hyper}, we consider hyperelliptic curves. Our results here are almost complete; they depend on those established in \cite{bmno} for general $g$ and improve on them for genus $6$. For the rest of the paper, we turn to non-hyperelliptic curves. In Section \ref{nonh}, following a general subsection, we obtain upper bounds in separate subsections for trigonal curves, smooth plane quintics and curves of Clifford index $2$. Section \ref{secexist} is concerned with existence results and this is followed in Section \ref{seck=n+1} with detailed results for the case $k=n+1$. In Section \ref{seclow}, we discuss bundles of ranks $2$ and $3$. Then, in Section \ref{secext}, we investigate extremal bundles, in other words those which maximise $k$ for given $n$ and $d$. Finally, in Section \ref{secBN}, we provide a graphical representation of our results.

Our methods are inspired in particular by those of \cite{bmno} and work of Mercat \cite{m1,m2} as well as our previous papers. 
Our thanks are due to the referee for a careful reading and some helpful suggestions.

\section{Background and some known results}\label{back}
 
Let $C$ be a smooth projective curve of genus $g$. Denote by $M(n,d)$ the moduli space of stable vector 
bundles of rank $n$ and degree $d$ and by $\widetilde M(n,d)$ the moduli space of 
S-equivalence classes of semistable bundles of rank $n$ and degree $d$. For any integer
$k \geq 1$, we define 
$$
B(n,d,k) := \{ E \in M(n,d) \; | \; h^0(E) \geq k \}
$$
and
$$
\widetilde B(n,d,k) := \{ [E] \in \widetilde M(n,d)  \;|\; h^0( \gr E) \geq k \},
$$
where $[E]$ denotes the S-equivalence class of $E$ and $\gr E$ is the graded object defined 
by a Jordan-H\"older filtration of $E$. The locus $B(n,d,k)$ has an {\it expected dimension}
\begin{equation}\label{eqbeta}
\beta(n,d,k):=n^2(g-1)+1-k(k-d+n(g-1)),
\end{equation}
known as the {\it Brill-Noether number}.
For any vector bundle $E$ on $C$, we write $n_E$ for the rank of $E$, $d_E$ for the degree of $E$ and $\mu(E)=\frac{d_E}{n_E}$ for the slope of $E$. 
The vector bundle $E$ is said to be {\it generated} if the evaluation map $H^0(E)\otimes {\mathcal O_C}\to E$ is surjective. We write also $h^0(E):=\dim H^0(E)$.

We recall the {\it dual span construction} (see, for example, \cite{bu} and \cite{m1}),
defined as follows. Let $L$ be a generated line bundle on $C$ with $h^0(L) \geq 2$.  Consider
the evaluation sequence
\begin{equation} \label{eq2.1}
0 \ra E_L^* \ra H^0(L) \otimes \cO_C \ra L \ra 0.
\end{equation}
Then $E_L$ is a bundle of rank $h^0(L) -1$ and degree $d_L$ with $h^0(E) \geq h^0(L)$.
 It is called the {\it dual span of} $L$ and is also denoted by $D(L)$. Although $E_L$ is not 
necessarily stable, this is frequently the case.

For any positive integer $r$, we define the {\it r-th gonality} of $C$ by
$$
d_r:= \min \{ d_L \;|\; L \; \mbox{ a line bundle on} \; C \; \mbox{with} \; h^0(L) \ge r+1 \}.
$$
The {\it Clifford index} of $C$ is defined by
$$\Cliff(C):=\min\{d_L-2(h^0(L)-1)\},$$
the minimum being taken over all line bundles $L$ with $d_L\le g-1$ and $h^0(L)\ge2$. Higher rank Clifford indices are then defined for all positive integers $n$ as follows:
$$\Cliff_n(C):=\frac1n\min\{d_E-2(h^0(E)-n)\},$$
the minimum being taken over all semistable bundles $E$ of rank $n$ with $d_E\le n(g-1)$ and $h^0(E)\ge2n$. We have always $\Cliff_n(C)\le \Cliff(C)$ \cite[Lemma 2.2]{ln} and
\begin{equation}\label{eqcliff}
\Cliff_n(C)=\Cliff(C) \mbox{ if } \Cliff(C)\le2
\end{equation}
(see \cite[Proposition 3.5]{ln}).

When $g=6$, we have $\Cliff(C)\le2$; moreover $d_5 = 10$, and $d_4 = 9$ if $C$ is non-hyperelliptic. If $\Cliff(C)=1$, then $C$ is either {\it trigonal} or a {\it smooth plane quintic}. In the first case, we denote  the unique trigonal bundle by $T$; in the second, we denote the hyperplane bundle by $H$. The gonalities $d_r$ for $r\le3$ are as follows:
\begin{equation} \label{e1}
C \; \mbox{trigonal}: \quad d_1 = 3, \, d_2= 6,\, d_3 = 7;
\end{equation}
\begin{equation} \label{e2}
C \; \mbox{smooth plane quintic}: \quad d_1 = 4, \, d_2= 5,\,  d_3 = 8;
\end{equation}
\begin{equation} \label{e3}
\Cliff(C) = 2: \quad d_1 = 4, \, d_2= 6, \, d_3 = 8.
\end{equation}
For $\Cliff(C) = 2$, this is obvious. For $C$ trigonal, see \cite[Remark 4.5(b)]{ln} and for $C$ a smooth plane quintic, see \cite[Theorem 2.1]{h}. We summarise some further useful properties for classical Brill-Noether loci on $C$ in the following proposition. Note that we call a curve $C$ {\it bielliptic} if $C$ is a double covering of an elliptic curve and $\Cliff(C)=2$. 

\begin{prop}\label{p2.0}
Let $C$ be a non-hyperelliptic curve of genus $g$. Then $B(1,4,2)\simeq B(1,6,3)\ne\emptyset$. Moreover,
\begin{enumerate}
\item[(i)] for $C$ general, $B(1,4,2)$ consists of precisely $5$ points;
\item[(ii)] if $\Cliff(C)=2$ and $C$ is not bielliptic, $B(1,4,2)$ is finite;
\item[(iii)]if $\Cliff(C)=1$ or $C$ is bielliptic, $\dim B(1,4,2)=1$.
\end{enumerate}
\end{prop}

\begin{proof}
The first assertion is an immediate consequence of the fact that $d_1\le4$ and Serre duality.
(i) follows from Castelnuovo's formula (see \cite[Formula (1.2), page 211]{acgh}. For (ii) and (iii), see Mumford's extension of Martens' Theorem \cite[Chapter IV, Theorem 5.2]{acgh}.
\end{proof}

\begin{rem}\label{reme}
{\rm If $\Cliff C=2$, then every $L\in B(1,4,2)$ is generated since $d_1=4$. The same holds for $C$ a smooth plane quintic, in which case $B(1,4,2)=\{H(-p)|p\in C\}$, where $H$ is the hyperplane bundle, by \cite[Theorem 2.1]{h}. For $C$ a trigonal curve, things are a little more complicated and are ruled by Maroni's theory (see \cite{ms} for a detailed account). Certainly $T(p)\in B(1,4,2)$ for all $p\in C$ and also $K_C\otimes T^{*2}\in B(1,4,2)$. In fact, it is easy to see that these are the only possibilities. For a general trigonal curve of genus $6$, $K_C\otimes T^{*2}$ is generated; these curves can be represented as $(3,4)$ curves on a quadric surface or as plane sextics having a triple point and a double point. However, there exist trigonal curves for which $K_C\otimes T^{*2}=T(p)$ for some $p\in C$ and is not generated. In the latter case, there are no generated bundles in $B(1,4,2)$.
}
\end{rem} 

\begin{prop} \label{p2.1}
Let $C$ be a curve, $n$ a positive integer and $L$ a line bundle on $C$ with $d_L = d_n$ and $h^0(L)=n+1$. Then $L$ is generated and, 
\begin{enumerate}
\item[(i)] if $\frac{d_p}{p} > \frac{d_n}{n}$ for all $p < n$, then $E_L$ is stable;
\item[(ii)] if $E$ is a semistable bundle of rank $n$ with $d_E < n \frac{d_p}{p}$ for all $p \leq n$, then $h^0(E) \le n$.
\end{enumerate}
\end{prop}

\begin{proof}
This follows from \cite[Propositions 4.9(e) and 4.11]{ln}.
\end{proof}

The following lemma is a restatement of \cite[Lemma 3.9]{pr} (see also \cite[Lemma 4.8]{ln}).

\begin{lem} \label{l2.2}
Let $C$ be a curve and $E$ a vector bundle of rank $n$ on $C$ with $h^0(E) \ge n+s$ for some $s \geq 1$. Suppose that $E$ has no proper subbundle $N$ with $h^0(N) > n_N$. Then $d_E
 \geq d_{ns}$. 
\end{lem}

In investigating the non-emptiness of $B(n,d,k)$ 
and $\widetilde B(n,d,k)$, it is sufficient by Serre duality and Riemann-Roch to consider the case $d\le n(g-1)$. 
For $g\le3$, a complete solution is known. 
For $g=4$ and $g=5$, partial results were obtained in \cite{ln1} and \cite{ln2}. In this paper, we investigate the case $g=6$. For future reference, we note some facts here. The first is \cite[Proposition 2.1]{ln2}.

\begin{prop}\label{pln2} Let $C$ be a curve of genus $g \ge 3$ and suppose $k\ge1$.
\begin{enumerate}
\item[(i)] If $0<d<2n$, then $\widetilde B(n,d,k)\ne\emptyset$ if and only if $k\leq n+\frac1g(d-n)$. Moreover
$B(n,d,k)\ne\emptyset$ under the same conditions except when $(n,d,k)=(n,n,n)$ with $n\ge2$.
\item[(ii)] If $C$ is non-hyperelliptic and $d=2n$, then $\widetilde B(n,d,k)\ne\emptyset$ if and only if $k\le\frac{g}{g-1}n$.
\item[(iii)] If $C$ is non-hyperelliptic and $d=2n$, then $B(n,d,k)\ne\emptyset$ if and only if $k\le\frac{g+1}{g}n$ or $(n,d,k)=(g-1,2g-2,g)$. Moreover $B(g-1,2g-2,g)=\{D(K_C)\}$.
\end{enumerate}
\end{prop}

\begin{cor} \label{c2.2}
Let $C$ be a non-hyperelliptic curve of genus $g$ and $s$ an integer, $s \ge 1$.
If $sn < d \le (s+1)n$ and $k \leq n+ \frac{1}{g}(d - sn)$ or $(n,d,k) = (g-1,(s+1)(g-1),g)$, then 
$$
B(n,d,k) \neq \emptyset.
$$
\end{cor}

\begin{proof}
Under the hypotheses on $d$ and $k$ in the statement, $B(n,d-(s-1)n,k)$ is non-empty by Proposition \ref{pln2}(i) and (iii). Tensoring by an effective line bundle of degree $s-1$ gives the result.
\end{proof}

\begin{cor} \label{c2.5}
Let $C$ be a non-hyperelliptic curve of genus $g$ and $F$ a generated vector bundle with $n_F \le g-2$ and $h^0(F^*) =0$. Then
$$
d_F > 2n_F.
$$
\end{cor}

\begin{proof}
The proof is by induction on the rank. If $n_F = 1$, then $d_F > 2$, since $C$ is non-hyperelliptic. Suppose $n_F \ge 2$ and the result holds for rank $< n_F$. 
If $F$ is semistable, the assertion holds by Proposition \ref{pln2}(i) and (ii). Otherwise, $F$ possesses a proper subbundle $G$ with $d_G > \frac{n_G d_F}{n_F}$.
Moreover, $F/G$ satisfies the hypotheses of the corollary. So by the inductive hypothesis, $d_{F/G} > 2 n_{F/G}$. Hence
$$
d_F > \frac{n_G d_F}{n_F} + 2 n_{F/G}
$$
which implies the assertion.
\end{proof}

\begin{cor}  \label{c2.6}
Let $C$ be a non-hyperelliptic curve of genus $g$. Let $F$ be a vector bundle with $n_F \le g-2$ such that $F$ and every subbundle have slope $\le 2$. Then
$$
h^0(F) \le n_F.
$$
\end{cor}

\begin{proof}
If $F$ is semistable, this follows from Proposition \ref{pln2}(i) and (ii). Otherwise, let $G$ be a subbundle of maximal slope. Then both $G$ and $F/G$ satisfy the hypotheses 
of the corollary. The result follows by induction.
\end{proof}

Our next result is the case $g=6$ of  \cite[Proposition 2.4]{ln2}. (Part (i) is contained in \cite{re}.)

\begin{prop} \label{p2.2}
Let $C$ be a non-hyperelliptic curve of genus $6$ and $E$ a semistable bundle on $C$ of rank $n$ and degree $d$.
\begin{enumerate}
 \item[(i)] If $1 \leq \mu(E) \leq 9$, then $h^0(E) \leq \frac{1}{2} (d+n)$.
 \item[(ii)] If $\mu(E) \geq 3$, then $h^0(E) \leq d-n$.
\end{enumerate}
\end{prop}

The following is \cite[Lemma 2.5]{ln2}.

\begin{lem} \label{l2.3}
Let $C$ be a curve. Suppose that $N$ is a generated line bundle on $C$ with $h^0(N) = 2$. Then, for any bundle $E$,
$$
h^0(N \otimes E) \geq 2h^0(E) - h^0(N^* \otimes E).
$$
In particular, if $E$ is either semistable with $\mu(E) < d_N$ or stable of rank $>1$ with $\mu(E) \leq d_N$, then
$$
h^0(N \otimes E) \geq 2h^0(E).
$$
\end{lem}

\begin{prop} \label{prop2.4}
Let $C$ be a trigonal curve of genus $6$ and $3n < d < 5n$.  
If $k \leq 2\left\lfloor n + \frac16(d -4n)\right\rfloor$ and $(n,d,k)\ne(n,4n,2n)$ or $(n,4n,2n-1)$, then 
$$
B(n,d,k) \neq \emptyset.
$$
\end{prop}

This is the case $g=6$ of \cite[Proposition 2.6]{ln2}.
We have the following similar result, valid for any curve of genus $6$ which admits a generated line bundle of degree $4$ (see Remark \ref{reme}).

\begin{prop} \label{p2.11}
Let $C$ be a curve of genus $6$ which admits a generated line bundle of degree $4$ and $4n < d \le 5n$. If $k \le  2\left\lfloor n + \frac16(d -5n)\right\rfloor$ and $(n,d,k) \neq (n,5n,2n)$ or $(n,5n,2n-1)$, then
$$
B(n,d,k) \neq \emptyset.
$$
\end{prop}

\begin{proof}
By Proposition \ref{pln2}, $B(n,d-4n, k') \neq \emptyset$ for $k' \leq n + \frac16(d-5n)$, except when $d=5n$ and $k' =n$ with $n \ge 2$. 
Now take $N$ in Lemma \ref{l2.3} to be a generated line bundle of degree $4$. The result follows.
\end{proof}

For curves of Clifford index $2$,  \cite[Theorem 2.1]{m} (see also \cite[Proposition 2.7]{ln2}) gives the following stronger version of Proposition \ref{p2.2}(i).

\begin{prop} \label{p1.6}
Let $C$ be a curve of genus $6$ with $\Cliff(C)=2$ and $E$ a semistable bundle on $C$ of rank $n$ and slope $\mu = \frac{d}{n}$. 
\begin{enumerate}
\item[(i)]
If $3 \leq \mu \leq 7$, then
$$
h^0(E) \leq \frac{d}{2}.
$$
\item[(ii)] 
If $1 \leq \mu \leq 3$, then 
$$
h^0(E) \leq n+\frac14(d -n).
$$
\end{enumerate}
\end{prop}

When $C$ is bielliptic, there is a 
partial converse to this proposition (see \cite[Theorem 5.3 and Proposition 5.4]{b} and \cite[Theorem 
3.1]{m}).

\begin{prop} \label{p1.7}
Let $C$ be a bielliptic curve and $n,d$ and $k$ positive integers.
\begin{enumerate}
\item[(i)] If $k \leq \frac{d}{2}$, then there exists a semistable bundle $E$ of rank $n$ and degree $d$ with $h^0(E) \geq k$.
\item[(ii)] If $k < \frac{d}{2}$, then there exists a stable bundle of rank $n$ and degree $d$ with 
$h^0(E) \geq k$.
\end{enumerate}
\end{prop}

A common method of construction is that of elementary transformations. We have in particular (see \cite[Th\'eor\`eme A.5]{m3})
\begin{prop} \label{p2.5}
 Let $C$ be a curve of genus $g\ge2$ and $L_1, \dots, L_n$ line bundles of degree $d$ on $C$ with $L_i \not \simeq L_j$ for $i \neq j$ and let $t > 0$. Then
 \begin{enumerate}
  \item[(i)] there exist stable bundles $E$ fitting into an exact sequence
  $$
  0 \ra L_1 \oplus \cdots \oplus L_n \ra E \ra \tau \ra 0
  $$
  where $\tau$ is a torsion sheaf of length $t$;
  \item[(ii)] there exist stable bundles $E$ fitting into an exact sequence
  $$
  0 \ra E \ra L_1 \oplus \cdots \oplus L_n \ra \tau \ra 0
  $$
  with $\tau$ as above.
 \end{enumerate}
\end{prop}

\begin{prop}   \label{p2.16}
Let $C$ be a curve of genus $6$. Then $B(n,d,k) \neq \emptyset$ for $4n < d < 5n$ and $k \le d-3n$. 
\end{prop}

\begin{proof}
Take $L_1, \dots, L_n$ to be pairwise non-isomorphic line bundles of degree 5 with $h^0(L_i) = 2$. Such bundles exist for all $n$ on any curve of genus 6. The result follows from Proposition
\ref{p2.5}(ii).
\end{proof}

\begin{prop} \label{p2.17}
Let $C$ be a
curve of genus $6$ such that either $\Cliff(C) = 1$ or $C$ is bielliptic. Then $B(n,d,k) \neq \emptyset$ in the following cases:
\begin{enumerate}
\item[(i)] $3n < d <4n$ and $k \le d -2n$;
\item[(ii)] $4n<d<5n$ and $k\le2n$.
\end{enumerate}
\end{prop}

\begin{proof}
For (i), the proof is the same as for Proposition \ref{p2.16} with $L_1, \dots , L_n$ pairwise non-isomorphic 
line bundles of degree 4 with $h^0(L_i) \ge 2$. These exist by Proposition \ref{p2.0}. (ii) is proved similarly using Proposition \ref{p2.5}(i).
\end{proof}

Another  useful way of constructing stable or semistable bundles with a specified number of independent sections (or proving they do not exist) is as extensions of bundles of  lower rank. The following lemma is the key to the use of this method (compare \cite[Lemma 2.8]{ln4}).

\begin{lem} \label{llift}
There exists a non-trivial extension of vector bundles
$$
0 \ra F \ra E \ra G \ra 0
$$
with the property that all sections of $G$ lift to $E$ if and only if the multiplication map
\begin{equation*} 
\mu: H^0(G) \otimes H^0(K_C \otimes F^*) \ra H^0(K_C \otimes F^* \otimes G)
\end{equation*}
is not surjective. Moreover, the non-trivial extensions for which all sections of $G$ lift are classified (up to scalar multiples) by ${\mathbb P}((\operatorname{Coker}\mu)^*)$.
\end{lem}

\begin{proof}
All sections of $G$ lift to $E$ if and only if the extension class is in the kernel of the canonical map
$$
H^1(G^* \otimes F) \ra \Hom(H^0(G),H^1(F)).
$$
The result follows from the fact that $\mu$ is the dual of this map.
\end{proof}

Finally, we have the following lemma (\cite[Lemma 2.10]{ln2}.

\begin{lem}\label{lbb}
If $\widetilde B(n,d,k)\ne\emptyset$, then $B(n',d',k')\ne\emptyset$ for some $(n',d',k')$ with $n' \leq n,\, \frac{d'}{n'}=\frac{d}{n}$ and $\frac{k'}{n'}\ge\frac{k}{n}$.
\end{lem}

\section{Hyperelliptic curves of genus $6$}\label{hyper}

For hyperelliptic curves, a complete solution for $0\le d\le4n$ is contained in \cite[Propositions 2.1, 2.2 and 2.3]{ln1} and \cite[Proposition 2.3]{ln2}.

\begin{prop} \label{p1.4}
Let $C$ be a hyperelliptic curve of genus $g \geq 3$ and $s$ an integer, $1 \le s \le g$. If
$(2s-2)n < d < 2sn$ and $\widetilde B(n,d,k) \neq \emptyset$, then
\begin{equation}\label{eqhyp}
k\le sn+\frac{s}{g}(d-(2s-1)n).
\end{equation}
\end{prop}

\begin{proof}
For $B(n,d,k)$, this is \cite[Theorem 6.2(1)]{bmno}. 
The semistable case follows by Lemma \ref{lbb}. 
\end{proof}

\begin{theorem}
Let $C$ be a hyperelliptic curve of genus $6$ and $4n < d < 5n$. Write $d = 5n - 6 \ell - \ell'$ with 
$1 \le \ell' \le 6$. 
\begin{enumerate}
\item[(i)] If $\widetilde B(n,d,k) \neq \emptyset$, then 
$$
k \le \left\{ \begin{array}{ccc} 
                 3n -3\ell -1 && \ell' = 1,\\
                 3n - 3 \ell -2 & if & \ell' = 2,\\
                 3n - 3 \ell -3 && \ell' \ge3;
                   \end{array}   \right.
$$
\item[(ii)] if $ k \le 3n - 3\ell -3, \; then \; B(n,d,k) \neq \emptyset$;
\item[(iii)] if $n\ge2$, $B(n,5n,k) \neq \emptyset \quad \Leftrightarrow \quad k \le 3n-1.$
\end{enumerate}
\end{theorem}

\begin{proof}
(i): Suppose $\widetilde B(n,d,k) \neq \emptyset$. By Proposition \ref{p1.4} with $s = 3$,
$$
k \le 3n  + \left\lfloor \frac{3}{6}(d -5n) \right\rfloor = 3n - 3 \ell + \left\lfloor - \frac{1}{2} \ell'
\right\rfloor.
$$
This gives (i) except if $\ell' = 2$, $3$ or $4$.
Suppose $E$ is a semistable bundle of rank $n$, degree $d$ and $h^0(E) = 3n - 3\ell -1$. Denoting by $H$ the hyperelliptic line bundle, we have, by Lemma \ref{l2.3}, 
\begin{eqnarray*}
6n - 6 \ell -2 = 2 h^0(E) & \le & h^0(H \otimes E) + h^0(H^* \otimes E)\\
& \le & 4n -4 \ell + \left\lfloor - \frac{4}{6} \ell' \right\rfloor + 2n-2 \ell + \left\lfloor - \frac{2}{6} \ell'
\right\rfloor
\end{eqnarray*}
by Proposition \ref{p1.4}. This is a contradiction for $\ell' \ge 2$.

Suppose $h^0(E) = 3n - 3\ell -2$. We now have
$$
6n - 6\ell -4 \le 6n - 6\ell  + \left\lfloor - \frac{4}{6} \ell' \right\rfloor + \left\lfloor - \frac{2}{6} \ell'
\right\rfloor.
$$
This is a contradiction for $ \ell' \ge 4$. For $\ell'=3$, apply Lemma \ref{l2.3} with $N=H$ and $E$ replaced by $H^*\otimes E$. We obtain
$$
2h^0(H^*\otimes E)\le h^0(E)+h^0(H^{*2}\otimes E)\le 3n-3\ell-2+n-\ell-1.
$$
So 
$$h^0(H^*\otimes E)\le 2n-2\ell-2.$$
Moreover, replacing $E$ by $H\otimes E$ in Lemma \ref{l2.3} gives
$$
2h^0(H\otimes E)\le h^0(H^2\otimes E)+h^0(E)\le 5n-5\ell+\left\lfloor -\frac56\ell'\right\rfloor+3n-3\ell-2.
$$
So, when $\ell'=3$,
$$h^0(H\otimes E)\le 4n-4\ell-3.
$$ 
Thus 
$$
6n-6\ell-4\le H^0(H\otimes E)+h^0(H^*\otimes E)\le 6n-6\ell-5,
$$ 
a contradiction. This completes the proof of (i).

(ii): By Proposition \ref{pln2}(i), $B(n,d -4n,k) \neq \emptyset$ if 
$$
k \le n + \left\lfloor \frac{1}{6}(d-5n) \right\rfloor = n - \ell + \left\lfloor -\frac{1}{6}\ell' \right\rfloor =n- \ell - 1.
$$
Let $E \in B(n,d-4n,n-\ell-1)$. Taking $N = H$ in Lemma \ref{l2.3}, we obtain
$$
h^0(H \otimes E) \ge 2h^0(E).
$$
Now apply Lemma \ref{l2.3} again with $N = H$ and $E$ replaced by $H \otimes E$. We obtain
$$
h^0(H^2 \otimes E) \ge 2 h^0(H \otimes E) - h^0(E) \ge 3 h^0(E).
$$
So $H^2\otimes E\in B(n,d,3n-3\ell-3)$, proving (ii)..

(iii) is proved in \cite[Corollary 6.1 and Proposition 6.1]{bmno}.
\end{proof}

\begin{rem}\label{rhyp}
{\rm This theorem shows that \eqref{eqhyp} is an almost exact upper bound for $k$. Moreover, if we write $\mu=\frac{d}{n}$ and $\lambda=\frac{k}{n}$, then for any given $\mu$, $\lambda$, we can choose $n$ and $d$ both divisible by $6$. This shows that  $\lambda$ attains the upper bound $3+\frac12(\mu-5)$ throughout the range $4<\mu<5$.
}
\end{rem}

\section{Upper bounds for non-hyperelliptic curves} \label{nonh}

\subsection{General results}

\begin{lem} \label{l3.4}
Let $C$ be a non-hyperelliptic curve of genus $6$  and $L = K_C(-p)$ for some 
$p \in C$.
Then $L$ is generated, $E_L$ is stable of rank $4$ and degree $9$ and $h^0(E_L) = 5$.
\end{lem}

\begin{proof}
The line bundle $L$ is generated, since $C$ is not hyperelliptic. Moreover, $h^0(L) =5$ by Riemann-Roch. Hence 
$E_L$ has rank 4 and degree 9 and is stable by \eqref{e1} -- \eqref{e3} and Proposition \ref{p2.1}(i). The fact that 
$h^0(E_L) \geq 5$ follows from dualizing the defining sequence \eqref{eq2.1}.

If $\Cliff(C) = 2$, we have $h^0(E_L) \leq 5$ by Proposition \ref{p1.6}(ii).
If $C$ is trigonal, then $h^0(T)=2$ and $h^0(K_C \otimes T^*) =4$. 
Hence $h^0(L \otimes T^*) \geq 3$ and there exists a non-zero homomorphism $T \ra L$.
Thus we obtain a non-zero homomorphism $D(L) \ra D(T)$, i.e. $E_L \ra T$. Since $E_L$ is stable,
this must be surjective and we have an exact sequence
$$
0 \ra F \ra E_L \ra T \ra 0.
$$
The rank-3 bundle $F$ is semistable. Hence $h^0(F) \leq \frac{18}{5}$ by Proposition 
\ref{pln2}(ii), which implies $h^0(E_L) \leq 5$.

Finally, suppose that $C$ is a smooth plane quintic. The rank-$2$ bundle $E_H$ is stable by \eqref{e2} and Proposition \ref{p2.1}(i), and $h^0(E_H) \le 3$ by \eqref{eqcliff} since $\Cliff(C) =1$. Arguing as above with $T$ replaced by $H$, we obtain an exact sequence 
$$
0 \ra F \ra E_L \ra E_H \ra 0.
$$
The rank-2 bundle $F$ is semistable. So $h^0(F) \leq \frac{12}{5}$ by Proposition \ref{pln2}(ii).
 This completes the proof.
\end{proof}

The following lemma is the case $g=6$ of \cite[Lemma 3.7(2)]{ln1}.

\begin{lem} \label{l3.3}
Let $C$ be a non-hyperelliptic curve of genus $6$ and $L = K_C(-p)$ for some $p \in C$.
Suppose $E$ is a bundle of rank $n$ and degree $d$ with $h^1(E \otimes L) = 0$ and 
$h^0(E) >n+ \frac{1}{5}(d-n)$. Then
$$
h^0(E_L^* \otimes E) > 0.
$$
\end{lem}

\begin{prop} \label{p3.5}
Let $C$ be a non-hyperelliptic curve of genus $6$ and $L = K_C(-p)$ for some $p \in C$.
Let $E$ be a semistable bundle of rank $n$ and degree $d$ with slope $\mu >2$. Suppose that
$h^0(E) >n+ \frac{1}{5}(d-n)$. Then
\begin{enumerate}
\item[(i)] $h^0(E_L^* \otimes E) > 0$;
\item[(ii)] $\mu > \frac{9}{4}$;
\item[(iii)] If $\mu < \frac{7}{3}$, then $E_L$ can be embedded as a subbundle of $E$.
\end{enumerate}
\end{prop}

\begin{proof}
This is included in \cite[Lemma 3.8]{ln1}, except that this gives only the weaker inequality 
$\mu \geq \frac{9}{4}$. If $\mu = \frac{9}{4}$, then by (iii), $E_L$ can be embedded as a
subbundle in $E$. If $n_E>4$, then $E/E_L$ satisfies the hypotheses of Proposition \ref{p3.5}.
So by induction every factor of a Jordan-H\"older filtration of $E$ is isomorphic to $E_L$.
Since $h^0(E_L) = 5$ by Lemma \ref{l3.4}, this contradicts the hypotheses.
\end{proof}

\begin{prop} \label{p3.4}
Let $C$ be a non-hyperelliptic curve of genus $6$. 
Suppose that $k = n + \frac{1}{5}(d-n)$. Suppose further that $2 < \frac{d}{n} \leq \frac{9}{4}$. 
If $B(n,d,k) \neq \emptyset$, then $(n,d,k) = (4,9,5)$. Moreover,
\begin{equation} \label{eq3.1}
B(4,9,5) = \{ E_L \;|\; L = K_C(-p) \; \mbox{for some} \; p \in C \}. 
\end{equation}
\end{prop}

\begin{proof} (This follows the same lines as \cite[Proposition 3.4]{ln2}.)
Suppose $E \in B(n,d,k)$. Note that we have $h^0(E) = k$ by Proposition \ref{p3.5}(ii).
We first claim that $E$ is generated.

If not, there exists an exact sequence 
$$
0 \ra F \ra E \ra \CC_q \ra 0
$$
with $h^0(F) = k$. Let $L = K_C(-p)$. Since $E \otimes L$ is stable with slope $> 11$, it follows that $E \otimes L$ is generated. Hence
$$
h^1(F \otimes L) = h^1(E \otimes L) = 0.
$$
It now follows from Lemma \ref{l3.3} that $h^0(E_L^* \otimes F) > 0$. Hence $E \simeq E_L$. This contradicts the assumption that $E$ is not generated.

It follows that  we have an exact sequence
$$
0 \ra G^* \ra H^0(E) \otimes \cO_C \ra E \ra 0
$$
with $n_G = k-n,\; d_G = d$ and $h^0(G) \geq k$. Hence $K_C \otimes G^*$ has rank $k-n$ and
\begin{eqnarray*}
 h^0(K_C \otimes G^*) &=& h^0(G) -d + 5(k-n)\\
 & \geq & k-d + d-n = n_G.\\
\end{eqnarray*}
Any such bundle necessarily has a section with a zero. So $K_C \otimes G^*$ admits a line subbundle $M$ with $h^0(M) \geq 1$ and $d_M \geq 1$ and we get the diagram
$$
\xymatrix{
0 \ar[r] & E^* \ar[d]_{\alpha} \ar[r] & W \otimes \cO_C \ar[d] \ar[r] & G \ar[r] \ar[d] & 0\\
0  \ar[r] & N^* \ar[r]  & V \otimes \cO_C  \ar[r] \ar[d] & K_C \otimes M^* \ar[r] \ar[d] & 0,\\
&  & 0 & 0 &
}
$$
where $W$ is a subspace of $H^0(G)$ of dimension $k$ and $V$ is the image of $W$ in $H^0(K_C \otimes M^*)$.
Since $d_{K_C \otimes M^*} \leq 9$, we have $\dim V \leq 5$ with equality only if $K_C \otimes M^* \simeq K_C(-p)$ for some $p \in C$. 

If $\alpha = 0$, then $E^*$ maps into $W' \otimes \cO_C$, where $W = W' \oplus V'$ and $V'$ maps isomorphically to $V$. It follows that $V' \otimes \cO_C$ 
maps to a trivial direct summand of $G$ contradicting the fact that $h^0(G^*) =0$. So $\alpha \neq 0$.

If $\dim V = 5$, we can write $K_C \otimes M^* \simeq K_C(-p) =: L$ and then $N \simeq E_L$, which is stable with $\mu(E_L) \geq \mu(E)$. Hence $E \simeq E_L$.

If $\dim V \leq 4$, then $N$ is generated and $h^0(N^*) = 0$. Moreover, $N$ possesses a quotient bundle $F$, which is a subsheaf of $E$. By Corollary \ref{c2.5}, $d_F > 2n_F$,
which contradicts the stability of $E$. This completes the proof.
 \end{proof}

 \begin{lem} \label{l3.5}
 Let $C$ be a non-hyperelliptic curve of genus $6$. Then
 \begin{enumerate}
  \item[(i)] $B(3,7,5) = \emptyset;$
  \item[(ii)] $\widetilde B(4,10,7) = \emptyset$.
 \end{enumerate}
\end{lem}

\begin{proof}
(i) follows from Lemma \ref{l2.2} and Corollary \ref{c2.6}.

(ii): Let $E$ be a semistable bundle of rank 4 and degree 10 with $h^0(E) \geq 7$. Since $B(2,5,4) = \emptyset$ by Proposition \ref{p2.2}(i), $E$ cannot be strictly semistable.

Applying Lemma \ref{l2.2}, it follows that $E$ possesses a subbundle $F$ with $h^0(F) > n_F$. The only possibility compatible with Corollary \ref{c2.6} is $n_F = 3,\; d_F =7$.

If $F$ is not stable, it possesses a subbundle which contradicts the stability of $E$. So $F$ is stable and by (i), $h^0(F) \leq 4$. 
Hence $E/F$ is a line bundle of degree 3 with $h^0(E/F) \geq 3$, contradicting the fact that $d_2\ge5$.
This completes the proof.
\end{proof}

 \begin{prop} \label{p3.6}
 Let $C$ be a non-hyperelliptic curve of genus $6$ and $E$ a semistable bundle of rank $n$ and degree $d$ with $\mu(E) > \frac{9}{4}$. Then
 $$
 h^0(E) \leq d-n.
 $$
 \end{prop}

 \begin{proof}
By Proposition \ref{p2.2}(ii) we can assume that $\mu(E) <3$. For $n \le 4$, the result follows by Proposition \ref{p2.2}(i) and Lemma \ref{l3.5}. 

Suppose now that $n \geq 5$ and the proposition is proved for rank $\le n-1$. As in the proof of \cite[Proposition 3.5]{ln2} there exists an exact sequence 
$$
0 \ra F \ra E \ra G \ra 0,
$$
where $F$ is a proper subbundle of maximal slope and is stable and $G$ is semistable. If $h^0(E) > d-n$, then $h^0(E) > n + \frac{1}{5}(d-n)$.
Hence, by Proposition \ref{p3.5}, there exists a non-zero homomorphism $\varphi: E_L \ra E$, where $L = K_C(-p)$ for some $p \in C$. The image of $\varphi$ generates a proper subbundle of $E$. So 
$$
\mu(F) \geq \mu(E_L) = \frac{9}{4}.
$$
Also $\mu(G) > \frac{9}{4}$ by semistability of $E$. If $\mu(F) > \frac{9}{4}$, then both $F$ and $G$ satisfy the inductive hypothesis and the inductive step is complete.
Otherwise $\mu(F) = \frac{9}{4}$ and $h^0(F) \le d_F - n_F$ by Proposition \ref{p3.5}.
Since $G$ still satisfies the inductive hypothesis, this is sufficient to prove the inductive step and hence the entire result.
\end{proof}

\begin{lem} \label{l3.7}
Let $C$ be a non-hyperelliptic curve of genus $6$. There exists a bundle $U\in B(2,5,3)$ if and only if $C$ is a smooth plane quintic.
Moreover, $U \simeq E_H$, where $H$ is the hyperplane bundle on the plane quintic.
\end{lem}

\begin{proof}
Let $E\in B(2,5,3)$. If $E$ is not generated, then a negative elementary transformation yields a semistable bundle of rank 2 and degree 4 with $h^0 = 3$ contradicting 
Proposition \ref{pln2}(ii). It follows that $E$ is isomorphic to $E_M$ for some line bundle $M$ of degree 5 with $h^0(M) = 3$. If $C$ is trigonal or $\Cliff(C) = 2$, we have
$d_2 = 6$ (see \eqref{e1} and \eqref{e3}), a contradiction.
If $C$ is a plane quintic, then $M \simeq H$ and $E \simeq E_H$. By Proposition \ref{p2.1}(i), $E_H$ is stable.
\end{proof}

\begin{lem} \label{l4.8}
Let $C$ be a non-hyperelliptic curve of genus $6$. Then
$$
B(3,7,4) \neq \emptyset \quad \Leftrightarrow \quad C \; \mbox{trigonal}.
$$
Moreover, if $C$ is trigonal, 
$$
B(3,7,4) = \{ E_M \;|\; M=K_C\otimes T^* \}.
$$
\end{lem}

\begin{proof}
If $C$ is not trigonal, then $d_3 = 8$ (see \eqref{e2} and \eqref{e3}). So $B(3,7,4) = \emptyset$ by Proposition \ref{p2.1}(ii). On the other hand, if $C$ is trigonal, then $d_3 = 7$. Moreover, by Serre duality, $B(1,7,4)=\{K_C\otimes T^*\}$.
By Proposition \ref{p2.1}(i), $E_M \in B(3,7,4)$ for $M=K_C\otimes T^*$. If $E\in B(3,7,4)$ is not generated, a negative elementary transformation yields a semistable bundle of rank 3, degree 6
with $h^0 = 4$, contradicting Proposition \ref{pln2}(ii). So $E$ is generated and the result follows.
\end{proof}

\subsection{Trigonal curves}  \label{sub4.2}
\begin{lem}\label{l4.9}
Let $C$ be a trigonal curve of genus $6$ and $E$ a semistable bundle of rank $n$ and 
degree $d$.
If $2n < d < 3n $, then
$$h^0(E)\le\frac56n+\frac5{18}d.$$
\end{lem}

\begin{proof}
Write $F:=K_C\otimes T^*\otimes E^*$. Then $d_F=7n-d$, so,  by Proposition \ref{p2.2}(i),
\begin{equation}\label{eq302}
h^0(F) \le \frac{1}{2}(7n -d +n) = 4n - \frac{d}{2}.
\end{equation}
By Lemma \ref{l2.3},
$$
h^0(T \otimes E) \geq 2 h^0(E).
$$
Hence by Riemann-Roch,
\begin{equation}\label{eq301}
2h^0(E)\le h^0(F)+ \chi(T \otimes E) \le 4n-\frac{d}{2}+d-2n=2n+\frac12d
\end{equation}
and
$$h^0(E)\le n+\frac14d.$$
We now use this to obtain an improved estimate for $h^0(F)$. In fact, by Lemma \ref{l2.3},
$$
h^0(T \otimes F) \ge 2h^0(F) - h^0(T^* \otimes F).
$$
Now $h^0(T\otimes F)=h^0(K_C\otimes E^*)=h^1(E)$, so
\begin{equation} \label{e4.2}
2h^0(F) \le h^0(E)-\chi(E)+h^0(T^*\otimes F).
\end{equation}
Since $n < d_{T^* \otimes F} < 2n$, we have, by Proposition \ref{pln2}(i),
\begin{equation*}
h^0(T^* \otimes F) \le n + \frac{1}{6}(d_F-4n)=n+\frac16(3n-d).
\end{equation*}
So, by \eqref{e4.2},
$$
2h^0(F)\le n+\frac14d-(d-5n)+n+\frac16(3n-d)=\frac{15}2n-\frac{11}{12}d
$$
and 
$$
h^0(F)\le \frac{15}4n-\frac{11}{24}d=4n-\frac12d-\frac1{24}(6n-d).
$$
Replacing \eqref{eq302} by this new inequality and substituting in \eqref{eq301}, we obtain
$$
h^0(E)\le n+\frac14d-\frac1{48}(6n-d).
$$
Iterating this calculation and taking account of \eqref{eq301} and \eqref{e4.2}, we obtain in the limit
\begin{eqnarray*}
h^0(E)&\le&n+\frac14d-\frac1{48}(6n-d)\left(1+\frac14+\frac1{16}+\cdots\right)\\
&=&n+\frac14d-\frac1{48}(6n-d)\frac43=\frac56n+\frac5{18}d.
\end{eqnarray*}
\end{proof}

\begin{prop} \label{p4.9}
Let $C$ be a trigonal curve of genus $6$ and $E$ a semistable bundle of rank $n$ and 
degree $d$.
If $2n \le d \le 3n $, then
\begin{equation}\label{eq201}
h^0(E) \le \left\{ \begin{array}{ccc}
                         n+\frac15(d-n)  & & 2n \le d \le \frac94n;\\
                         d-n & if & \frac94n < d \le \frac{33}{13}n;\\
                        \frac56n + \frac5{18}d && \frac{33}{13}n \le d <3n;\\
		      2n && d=3n.
                        \end{array} 
                        \right.
\end{equation}
\end{prop}

\begin{proof}
This follows from Propositions \ref{pln2}(ii), \ref{p3.5} and \ref{p3.6}, Lemma \ref{l4.9} and the fact that $\Cliff_n(C)=1$ (see \eqref{eqcliff}).
\end{proof}

\begin{rem} \label{r4.10}
{\rm If $d = 3n$ and $E$ is stable, $E\not\simeq T$, the argument of Lemma \ref{l4.9} gives
$$
h^0(E) \le \frac53n.
$$
}
\end{rem}

We are not able to obtain a bound smaller than that of Proposition \ref{p2.2} in the range $3n<d\le4n$ for an arbitrary trigonal curve. However, when $C$ admits a generated bundle in $B(1,4,2)$, we can do so.

\begin{prop}  \label{pr4.11}
Let $C$ be a trigonal curve of genus $6$ admitting a generated bundle $Q\in B(1,4,2)$ and $E$ a 
semistable bundle of rank $n$ and degree $d$. If $3n <d\le 4n$, then
\begin{equation}\label{eq202}
h^0(E) \le \left\{ \begin{array}{ccc}
                         \frac{3}{4}n + \frac{13}{36} d & & 3n < d \le \frac{45}{13}n;\\
                         2n & if & \frac{45}{13}n \le d < \frac{15}{4}n;\\
                        \frac{1}{2}n + \frac{2}{5}d && \frac{15}{4}n \le d \le 4n.
                        \end{array} 
                        \right.
\end{equation}
\end{prop}

\begin{proof}
Since $2 \le \mu(K_C \otimes Q^* \otimes E^*) < 3$, we have
$$
h^0(K_C \otimes Q^* \otimes E^*) \le \left\{ \begin{array}{ccc}
                                                          \frac{5}{2}n - \frac5{18}d  && 3n < d \le \frac{45}{13}n;\\
                                                          5n-d  & if & \frac{45}{13}n \le d < \frac{15}{4}n;\\
                                                          2n - \frac{1}{5}d  && \frac{15}{4}n \le d \le 4n.
                                                          \end{array}
                                                             \right.
$$
by Proposition \ref{p4.9}. By Lemma \ref{l2.3}, if $d < 4n$ or $d = 4n$ 
and $E$ is stable, $E\not\simeq Q$, 
$$
 h^0(Q \otimes E) \ge 2h^0(E).
$$
Hence
$$
2h^0(E) \le h^0(K_C\otimes Q^* \otimes E^*) +  \chi(Q \otimes E) = h^0(K_C \otimes Q^* \otimes E^*)+ d - n.
$$
The result follows.

If $E\simeq Q$, then $h^0(E) = 2 < \frac{1}{2}+ \frac{8}{5}$. The result for semistable bundles
with $d = 4n$ follows from Lemma \ref{lbb}. 

\end{proof}

\begin{prop} \label{p4.11}
Let $C$ be a trigonal curve of genus $6$ and $E$ a semistable bundle of rank $n$ and degree $d$. If $4n \le d \le 5n$, then 
\begin{equation}\label{eq203}
h^0(E) \leq \left\{ \begin{array}{ccc}
                    \frac59n + \frac49d &  & 4n \le d < \frac{58}{13}n;\\
                    \frac{13}{6}n + \frac{1}{12}d & if & \frac{58}{13}n \le d < \frac{19}{4}n;\\
                    \frac{4}{15}n + \frac{29}{60}d&  & \frac{19}{4}n \le d \le 5n.                    
                    \end{array} \right.
\end{equation}                  
\end{prop}

\begin{proof}
Suppose first that $4n<d<5n$. Since $2<\mu(K_C\otimes T^*\otimes E^*)<3$, we have, by Proposition \ref{p4.9},
\begin{equation}\label{eq100}
h^0(K_C \otimes T^* \otimes E^*) \le \left\{ \begin{array}{ccc}
                                              \frac{25}9n - \frac5{18}d && 4n < d < \frac{58}{13}n;\\
                                              6n-d & if & \frac{58}{13}n \le d < \frac{19}{4}n;\\
                                              \frac{11}{5}n - \frac{1}{5}d && \frac{19}{4}n \le d < 5n.
                                             \end{array}
                                            \right.
\end{equation}
By Lemma \ref{l2.3},
$$
h^0(T \otimes E) \ge 2h^0(E) - h^0(T^* \otimes E).
$$
Hence 
\begin{equation} \label{e4.2a}
2h^0(E) \le h^0(K_C \otimes T^* \otimes E^*)+\chi(T\otimes E)+h^0(T^*\otimes E).
\end{equation}
We have $n < d_{T^* \otimes E} < 2n$. So by Proposition \ref{pln2}(i),
\begin{equation}\label{eq101}
h^0(T^* \otimes E) \le n + \frac{1}{6}(d-4n).
\end{equation}
Introducing the inequalities for $h^0(K_C \otimes T^* \otimes E^*)$ and $h^0(T^* \otimes E)$ into \eqref{e4.2a} completes the proof for $4n<d<5n$.

If $d=4n$ and $E$ is stable,  then $K_C\otimes T^*\otimes E^*$ is a stable bundle of slope $3$, so,  by Remark \ref{r4.10}, $h^0(K_C\otimes T^*\otimes E^*)\le\frac53n$ unless $E\simeq K_C\otimes T^{*2}$; moreover \eqref{eq101} remains valid. If $E\simeq   K_C\otimes T^{*2}$, then $h^0(E)=2<\frac{21}9$. The result for semistable bundles with $d=4n$ follows from Lemma \ref{lbb}.

Finally, if $d=5n$, then $h^0(K_C\otimes T^*\otimes E^*)\le \frac65n$ by Proposition \ref{pln2}(ii); moreover, if $E$ is stable, then, by Proposition \ref{pln2}(iii),  \eqref{eq101} holds unless $T^*\otimes E\simeq D(K_C)$. In this case, $h^0(T^*\otimes E)=6$ and $h^0(K_C\otimes T^*\otimes E^*)\le6$ by Proposition \ref{pln2}(ii). It follows from \eqref{e4.2a} that $2h^0(E)\le27$, so $h^0(E)\le13<\frac{161}{12}$. The result for $d=5n$ is completed using again Lemma \ref{lbb}.
\end{proof}

We summarise the results of this subsection in the following theorem.

\begin{theorem}\label{t201}
Let $C$ be a trigonal curve of genus $6$. If $2n<d\le5n$ and $\widetilde{B}(n,d,k)\ne\emptyset$, then one of the following holds,
\begin{enumerate}
\item[(i)] $2n<d\le3n$ and $k$ satisfies the inequality of \eqref{eq201};
\item[(ii)] $3n<d<4n$ and $k\le\frac12(d+n)$;
\item[(iii)] $4n\le d\le5n$ and $k$ satisfies the inequality of \eqref{eq203}.
\end{enumerate}
If in addition $C$ admits a generated bundle in $B(1,4,2)$, then {\rm (ii)} can be replaced by
\begin{enumerate}
\item[(ii)$'$] $3n<d\le4n$ and $k$ satisfies the inequality of \eqref{eq202}.
\end{enumerate}
\end{theorem}

\begin{proof} This follows from Propositions \ref{p2.2}(i), \ref{p4.9}, \ref{pr4.11} and \ref{p4.11}.
\end{proof}
 
\subsection{Smooth plane quintics}\label{sub4.3}

\begin{prop} \label{prop4.12}
Let $C$ be a smooth plane quintic and $E$ a semistable bundle of rank $n$ and degree $d$. If $2n \le d \le3n$, then
\begin{equation}\label{eq204}
h^0(E) \le \left\{ \begin{array}{ccc}
                         n+\frac15(d-n)  & & 2n \le d \le \frac94n;\\
                         d-n & if & \frac94n < d \le \frac73n;\\
                         \frac43n && \frac73n < d <\frac52n;\\
		      d-n && \frac52n\le d\le3n.
                        \end{array} 
                        \right.
\end{equation}
\end{prop}

\begin{proof}
Except for the range $\frac73n<d<\frac52n$, this follows from Propositions \ref{pln2}(ii), \ref{p3.5} and \ref{p3.6}. Suppose therefore that $\frac73n<d<\frac52n$. Then certainly, since $d_{K_C\otimes H^*\otimes E^*}=5n-d$, we have by Proposition \ref{p3.6},
$$
h^0(K_C \otimes H^* \otimes E^*) \le 5n-d-n = 4n-d.
$$
Consider the sequence
$$
0 \ra E_H^* \ra H^0(H) \otimes \cO_C \ra H \ra 0.
$$
Tensoring by $E$, taking global sections and noting that $h^0(E_H^* \otimes E) = 0$, we obtain
\begin{eqnarray*}
3 h^0(E)  \le h^0(H \otimes E)& = & h^0(K_C\otimes H^* \otimes E^*) + \chi(H \otimes E)\\
& \le & 4n-d+d =4n.
\end{eqnarray*}
This completes the proof.
\end{proof}

\begin{rem}
{\rm
If $d = \frac{5}{2}n$ and $E$ is stable, we still have $h^0(E_H^* \otimes E) = 0$, unless $E \simeq E_H$. So $E_H$ is the only stable bundle 
of slope $\frac{5}{2}$ for which $h^0 > \frac{4}{3}n$.
}
\end{rem}

\begin{prop}  \label{pr4.12}
Let $C$ be a smooth plane quintic and $E$ a semistable bundle of rank $n$ and degree $d$.
If $3n \le d \le 4n$, then 
\begin{equation}\label{eq205}
h^0(E) \le \left\{ \begin{array}{ccc}
                   2n && 3n \le d \leq \frac{7}{2}n;\\
                   \frac{1}{6}n + \frac{1}{2}d & if & \frac{7}{2}n < d \leq \frac{11}{3}n;\\
                   2n &  & \frac{11}{3}n \le d \le \frac{15}{4}n;\\
                   \frac{1}{2}n + \frac{2}{5}d && \frac{15}{4}n \le d \le 4n.
                   \end{array}
                   \right.
\end{equation}
\end{prop}

\begin{proof}
Let $Q\in B(1,4,2)$. The bundle $Q$ is necessarily generated (see Remark \ref{reme}). Since $2 \le \mu(K_C \otimes Q^* \otimes E^*) \le 3$, we have by Proposition \ref{prop4.12},
$$
h^0(K_C \otimes Q^* \otimes E^*) \leq \left\{ \begin{array}{ccc} 
                                              5n-d && 3n \le d \le \frac{7}{2}n;\\
                                              \frac{4}{3}n & if & \frac{7}{2}n < d \le \frac{11}{3}n;\\
                                              5n-d &  & \frac{11}{3}n \le d \le \frac{15}{4}n;\\
                                              2n-\frac{1}{5}d && \frac{15}{4}n \le d \le 4n.
                                              \end{array}
                                              \right.
$$
 The rest of the proof is identical to the proof of Proposition \ref{pr4.11}.
\end{proof}

\begin{prop} \label{p4.13}
Let $C$ be a smooth plane quintic and $E$ a semistable bundle of rank $n$ and degree $d$. If
$4n \le d\le 5n$, then
\begin{equation}\label{eq206}
h^0(E) \le \left\{ \begin{array}{ccc}
                   \frac{2}{3}d - \frac{1}{3}n & & 4n \le d < \frac{9}{2}n;\\
                   \frac{8}{3}n & if & \frac{9}{2}n \le d < 5n;\\
                   3n&&d=5n.
                   \end{array}
                   \right.
\end{equation}
\end{prop}

\begin{proof} 
Since $\frac52 \le \mu(K_C \otimes E_H^* \otimes E^*) \le \frac72$, we have by Propositions \ref{prop4.12} and \ref{pr4.12},
$$
h^0(K_C \otimes E_H^* \otimes E^*) \le \left\{ \begin{array}{ccc}
                                                4n & if & 4n \le d <\frac{9}{2}n;\\
                                                13n -2d & if & \frac{9}{2}n \le d \le5n.
                                               \end{array}
                                               \right.
$$ 
Now consider the sequence
$$
0 \ra H^* \ra H^0(E_H) \otimes \cO_C \ra E_H \ra 0.
$$
Tensoring by $E$, taking global sections and noting that $h^0(H^* \otimes E) = 0$ if $d<5n$, we obtain
$$
h^0(E_H \otimes E) \geq 3h^0(E).
$$
So 
\begin{eqnarray*}
3h^0(E) & \le & h^1(E_H \otimes E) + \chi(E_H \otimes E) \\
& \le & \left\{ \begin{array}{ccc}
               4n + 2d - 5n & if & 4n \le d < \frac{9}{2}n;\\
               13n - 2d + 2d - 5n & if & \frac{9}{2}n \le d < 5n
               \end{array}
               \right. \\
& = & \left\{ \begin{array}{ccc} 
              2d - n & if & 4n \le d < \frac{9}{2}n; \\
              8n & if & \frac{9}{2}n \le d < 5n.
              \end{array}
              \right.
\end{eqnarray*}
This completes the proof for $d<5n$. If $d=5n$, since $\Cliff_n(C)=1$ (see \eqref{eqcliff}), it follows from the definition of $\Cliff_n(C)$ that $h^0(E)\le3n$.
\end{proof}

\begin{rem} \label{rem4.19}
{\rm
The same proof shows that, if $E$ is stable with $d = 5n$ and $E  \not \simeq H$, then 
$h^0(E) \leq \frac{8}{3}n$.
}
\end{rem}

We summarise the results of this subsection in the following theorem.

\begin{theorem}\label{t202}
Let $C$ be a smooth plane quintic. If $2n<d\le5n$ and $\widetilde{B}(n,d,k)\ne\emptyset$, then one of the following holds,
\begin{enumerate}
\item[(i)] $2n<d\le3n$ and $k$ satisfies the inequality of \eqref{eq204};
\item[(ii)] $3n<d\le4n$ and $k$ satisfies the inequality of \eqref{eq205};
\item[(iii)] $4n< d\le5n$ and $k$ satisfies the inequality of \eqref{eq206}.
\end{enumerate}
\end{theorem}

\begin{proof} This follows from Propositions \ref{prop4.12}, \ref{pr4.12} and \ref{p4.13}.
\end{proof}

\subsection{Curves of Clifford index $2$}

\begin{prop} \label{p3.7}
Let $C$ be a curve of genus $6$ with $\Cliff(C) = 2$ and $E$ a semistable bundle of rank $n$ and degree $d$ with $\mu(E)>\frac{9}{4}$. Then 
$$
h^0(E) < d-n.
$$
\end{prop}

\begin{proof}
In view of Proposition \ref{p1.6} and Lemma \ref{l4.8}, there are no semistable bundles $E$ of rank $n_E\le4$ with $ \mu(E)>\frac{9}{4}$ and $h^0(E) \geq d_E-n_E$. 

The proof proceeds exactly as for that of Proposition \ref{p3.6}. Note that the assumption $h^0(E)\ge d-n$ is sufficient to give $h^0(E)>n+\frac15(d-n)$, so that Proposition \ref{p3.5} applies where required.
\end{proof}

\begin{theorem} \label{th3.18}
 Let $C$ be a curve of genus $6$ with $\Cliff(C) = 2$. If $2n < d\leq 5n$ and 
$\widetilde B(n,d,k) \neq \emptyset$, then one of the following holds,
\begin{enumerate}
\item[(i)] $2n <d \leq \frac{9}{4}n$ and $k \leq n + \frac{1}{5}(d-n)$;
\item[(ii)] $\frac{9}{4}n < d \leq \frac{7}{3}n$ and $k < d-n$;
\item[(iii)] $\frac{7}{3}n < d \leq 3n$ and $k \leq n + \frac{1}{4}(d-n)$;
\item[(iv)] $3n < d \leq 5n$ and $k \leq \frac12d$.
\end{enumerate}
\end{theorem}

\begin{proof}
(i) follows from Proposition \ref{p3.5}, (ii) from Proposition \ref{p3.7} and (iii) and (iv) from Proposition 
\ref{p1.6}.
\end{proof}

\begin{rem}
{\rm
For $C$ bielliptic, the bound (iv) is sharp by Proposition \ref{p1.7}.
}
\end{rem}

\section{Existence results for non-hyperelliptic curves}\label{secexist}

\begin{prop}\label{p3.11}
Let $C$ be a non-hyperelliptic curve of genus $6$. Then $B(n,d,k) \neq \emptyset$ in the following cases:
\begin{enumerate}
 \item[(i)] $(n,d,k)= (5r+s, 10r+2s+1, 6r+s) \; \mbox{for} \; 1 \leq r \leq 5, \; s \geq 0$;
 \item[(ii)] $(n,d,k) = (5r+s,10r+2s+2, 6r+s)$ for $1 \leq r \leq 5, \; s \geq 5r+1$;
 \item[(iii)] $(n,d,k) = (6r+s, 12r+2s+1,7r+s)$ for $r \geq 1, \; s \geq 0$.
\end{enumerate}
\end{prop}

\begin{proof}
These are special cases of \cite[(3.3), Proposition 3.6 and Example 3.9]{ln1}. 
\end{proof}

\begin{prop} \label{p3.10}
 Let $C$ be a non-hyperelliptic curve of genus $6$. Then $B(2,6,4) = \emptyset$. Moreover,
 \begin{enumerate}
 \item[(i)] if $\Cliff(C) = 2$, then
 $$
 B(2,6,3) = \{ E_M \;|\; M \in B(1,6,3) \}\neq \emptyset;
 $$
\item[(ii)] if $C$ is trigonal, then $B(2,6,3)=\emptyset$;
 \item[(iii)] If $C$ is a smooth plane quintic with hyperplane bundle $H$, then $E \in B(2,6,3)$ if and only if there exists an exact sequence
 $$
 0 \ra E_H \ra E \ra \CC_p \ra 0 \quad \mbox{for some} \; p \in C.
 $$
 \end{enumerate}
\end{prop}

\begin{proof}
If $B(2,6,4) \neq \emptyset$, then certainly $\Cliff_2(C)=\Cliff(C) =1$. We obtain a contradiction by \cite[Corollary 4.8]{ln3} if $C$ is trigonal and by 
\cite[Proposition 3.1]{ln4} if $C$ is a smooth plane quintic. 

(i): Since $d_2 = 6$ (see \eqref{e3}), every $M \in B(1,6,3)$ is generated. Hence $E_M$ exists and belongs to $B(2,6,3)$ by Proposition \ref{p2.1}(i).

Now suppose $E \in B(2,6,3)$. If $E$ is not generated, a negative elementary transformation yields a stable bundle of rank 2, degree 5 with $h^0= 3$, contradicting Lemma \ref{l3.7}. So $E$ 
is generated and there exists an exact sequence
$$
0 \ra M^* \ra H^0(E) \otimes \cO_C \ra E \ra 0.
$$
So $M \in B(1,6,3)$ and $E \simeq E_M$. All such $E_M$ are stable by Proposition \ref{p2.1}(i).

(ii): Suppose $E\in B(2,6,3)$. Since $d_2=6$ (see \eqref{e1}), it follows as in (i) that $E\simeq E_M$ for some $M\in B(1,6,3)$. However, in this case, Proposition \ref{p2.1}(i) does not apply. We shall show in fact that $E_M$ cannot be stable. For this, we consider the exact sequence
$$
0\ra E_M^*\ra H^0(M)\otimes {\mathcal O}_C\ra M\ra0.
$$
Tensoring by $T$ and taking global sections, we obtain an inequality
$$
h^0(T\otimes E_M^*)\ge 3 h^0(T)-h^0(T\otimes M).
$$
Since $h^0(T)=2$ and $h^0(T\otimes M)\le5$, it follows that there exists a non-zero homomorphism $E_M\to T$. This contradicts the stability of $E_M$.

(iii): First note that, by Hartshorne's version of Noether's theorem \cite[Theorem 2.1]{h}, there are no generated line bundles in $B(1,6,3)$. It follows that there
are no generated bundles in $B(2,6,3)$. Then the assertion follows from Lemma \ref{l3.7}.
\end{proof}

\begin{cor} \label{c4.11}
Let $C$ be a non-hyperelliptic curve of genus $6$. Then
$$
B(2r,6r-1,3r-1) \neq \emptyset \quad \mbox{for} \quad 1 \leq r \leq 5.
$$
If either $\Cliff(C) =1$ or $C$ is bielliptic, the same holds for every positive integer $r$.
\end{cor}

\begin{proof}
Suppose first that $E_1, \dots, E_r$ are pairwise non-isomorphic bundles in $B(2,6,3)$. Let $E$ be an elementary transformation
\begin{equation}\label{eqelem}
0 \ra E \ra E_{1} \oplus \cdots \oplus E_{r} \ra \CC_p \ra 0
\end{equation}
for some $p \in C$ such that the homomorphisms $E_{i} \ra \CC_p$ are all non-zero. Since the partial direct sums of the $E_{i}$ are the only subbundles of 
$E_{1} \oplus \cdots \oplus E_{r}$ of slope 3, it follows that every subbundle of $E$ has slope $< 3$. Hence $E \in B(2r,6r-1,3r-1)$.

For $C$ general, there exist exactly 5 line bundles in $B(1,6,3)$ by Proposition \ref{p2.0}(i). Hence by Proposition \ref{p3.10}(i), there
exist 5 pairwise non-isomorphic bundles $E_i \in B(2,6,3)$. This proves the first part of the corollary for a general curve and hence for any curve by semicontinuity.

If $C$ is a smooth plane quintic, then $B(2,6,3)$ is infinite by Proposition \ref{p3.10}(iii) and the second part of the corollary follows. If $C$ is bielliptic, then $B(1,6,3)$ is infinite by Proposition \ref{p2.0}(iii).
Moreover, since $d_2=6$ by \eqref{e3}, all bundles $M \in B(1,6,3)$ are generated and the corresponding bundles $E_M$ are stable by Proposition \ref{p2.1}(i).

For $C$ trigonal, $B(2,6,3)=\emptyset$. In this case, we consider extensions
\begin{equation}\label{eqtrig}
0\ra L\ra E\ra T\ra0,
\end{equation}
with $L$ a line bundle with $d_L=3$, $h^0(L)=1$. In fact, since $h^0(K_C\otimes T^{*2})=2$, we can take $L=K_C\otimes T^{*2}(-q)$ with $q\in C$. When $K_C\otimes T^{*2}\simeq T(p)$ (see Remark \ref{reme}), we take $q\ne p$.
We then consider the multiplication map
$$ \mu:H^0(T)\otimes H^0(K_C\otimes L^*)\lra H^0(T\otimes K_C\otimes L^*).$$
We have $h^0(T\otimes K_C\otimes L^*)=5$ by Riemann-Roch, since $L\not\simeq T$. Moreover $h^0(T)=2$ and $h^0(K_C\otimes L^*)=3$ by Serre duality and Riemann-Roch. Since $T$ is generated, we have an exact sequence
$$0\ra T^*\ra H^0(T)\otimes\mathcal{O}_C\ra T\ra0.$$
Tensoring by $K_C\otimes L^*$ and taking sections, we see that $\operatorname{Ker}\mu=H^0(T^*\otimes K_C\otimes L^*)$, so
$$\dim\operatorname{Ker}\mu=h^0(T^*\otimes K_C\otimes L^*)=h^0(T(q))=2.$$
It follows from Lemma \ref{llift} that there is a unique non-trivial extension (up to scalar multiples) \eqref{eqtrig} for which all sections of $T$ lift. Note further that $E$ is semistable, that $L$ is the unique line subbundle of $E$ of degree $3$ and there is no non-zero homomorphism $E\to L$. Moreover, if $E_1, \ldots, E_r$ are  bundles constructed in this way using pairwise  non-isomorphic line bundles $L_1, \ldots, L_r$, then there are no non-zero homomorphisms  $E_i\to E_j$ for $i\ne j$. It follows that $E_1\oplus \ldots \oplus  E_r$ admits finitely many subbundles of slope $3$, namely direct sums of some $L_{i_k}$ and some $E_{j_\ell}$ with $i_k\ne j_\ell$ for all $k,\ell$. If we now take \eqref{eqelem} to be an extension such that the restriction of the homomorphism $E_1\oplus\ldots\oplus E_r\to{\mathbb C}_p$ to each of these subbundles is non-zero, it follows that $E$ is stable and we are done.
\end{proof}

\begin{cor} \label{c4.12}
Let $C$ be a non-hyperelliptic curve of genus $6$. Then
$$
B(2r,6r+1,3r) \neq \emptyset \quad \mbox{for} \quad 1 \leq r \leq 5.
$$
If either $\Cliff(C) =1$ or $C$ is bielliptic, the same holds for every positive integer $r$.
\end{cor}

\begin{proof}
We consider extensions 
$$
0 \ra E_{1} \oplus \cdots \oplus E_{r} \ra E \ra \CC_p \ra 0
$$
where the $E_{i}$ are as in the previous proof. The general such extension gives a stable bundle $E$. Hence $B(2r,6r+1,3r) \neq \emptyset$ and the result follows as in the previous proof.
\end{proof}

\begin{rem}
{\rm
Note that, for the bundles with $r=4, 5$ in Corollary \ref{c4.12}, the Brill-Noether number is negative.
}
\end{rem}

\begin{prop} \label{p4.18}
Let $C$ be a curve of genus $6$. Then
$$
\widetilde B(2,10,5) \neq \emptyset \quad \mbox{and} \quad B(3,10,5) \neq \emptyset.
$$
Moreover, 
\begin{enumerate}
\item[(i)] if $\Cliff(C) = 2$ or $C$ is trigonal, then
$$
B(2,10,5) = \widetilde B(2,10,5) \neq \emptyset;
$$
\item[(ii)] if $C$ is a smooth plane quintic, then 
$B(2,10,5)=\emptyset$.
\end{enumerate}
\end{prop}

\begin{proof}
For $C$ general, both $B(2,10,5)$ and $B(3,10,5)$ are non-empty by \cite[Propositions 4.1 and 4.4]{n}. It follows that $\widetilde B(2,10,5) \neq \emptyset$ and 
$B(3,10,5) \neq \emptyset$ for any $C$ by semicontinuity. 

If $\Cliff(C) = 2$ or $C$ is trigonal, then $B(1,5,3) = \emptyset$. So $\widetilde B(2,10,5) = B(2,10,5)$. This completes the proof of (i).

If $C$ is a smooth plane quintic and $E\in B(2,10,5)$, then it follows from Lemma \ref{l2.2} that there is an exact sequence
\begin{equation}\label{eq2105}
0\lra M\lra E\lra N\lra0,
\end{equation}
with $M\in B(1,4,2)$ and $N\in B(1,6,3)$; so $M\simeq H(-p)$ and $N\simeq H(q)$ for some $p, q\in C$. Moreover, all sections of $N$ lift to $E$.  Now consider the multiplication map
$$\mu:H^0(N)\otimes H^0(K_C\otimes M^*)\lra H^0(N\otimes K_C\otimes M^*).$$
Since $K_C\otimes M^*\simeq H^2\otimes M^*\simeq H(p)$ and both $H^0(H(p))$ and $H^0(H(q))$ are isomorphic to $H^0(H)$, it follows from the surjectivity of the map $H^0(H)\otimes H^0(H)\to H^0(H^2)$ that $\operatorname{Im}\mu$ maps onto the subspace $H^0(H^2)$ of $H^0(H(p)\otimes H(q))$. This subspace has codimension $1$, so (up to scalar multiples) there exists by Lemma \ref{llift} a unique non-trivial extension \eqref{eq2105} for which all sections of $N$ lift. On the other hand, there exists an extension \eqref{eq2105} such that the pullback by a non-zero homomorphism $H\to H(q)$ splits. For this extension, since $H^0(H)$ maps isomorphically to $H^0(N)$ and $H\subset E$, it follows that all sections of $N$ lift. So this must be the unique non-trivial extension with this property. In particular, $E$ is not stable, a contradiction. This completes the proof of (ii). 
\end{proof}

\begin{rem}
{\rm
The proof of \cite[Proposition 4.4]{n} shows that $B(2,10,5)$ is isomorphic to $B(3,10,5)$ for any curve of genus 6 with $\Cliff(C) = 2$. For $C$ general, the BN-loci consist of a single point.}
\end{rem}

\begin{rem}
{\rm 
The Brill-Noether number of the bundles of Proposition \ref{p4.18} is negative.
}
\end{rem}

\begin{prop} \label{p4.20}
Let $C$ be a non-hyperelliptic curve of genus $6$, $n$ and $s$ positive integers, $k\le2n$ . Then
$$
B(n,4n+s,k) \neq \emptyset
$$
in the following cases.
\begin{enumerate}
\item[(i)] $C$ general, $n\le 5$;
\item[(ii)] $\Cliff(C)=2$, $n\le5$, $\gcd(n,s)=1$;
\item[(iii)] $\Cliff(C) = 1$ or $C$ bielliptic, any $n$.
\end{enumerate} 
\end{prop}

\begin{proof}
Let $L_1,\ldots,L_n$ be distinct line bundles in $B(1,4,2)$. (i) and (iii) follow from Propositions \ref{p2.0} and \ref{p2.5}(i), combined with the fact that $B(n,d,k-1)\supset B(n,d,k)$. (ii) now follows from (i) by semicontinuity.
\end{proof}

\begin{rem}
{\rm
For $n = 3,4,5$, the Brill-Noether number of the bundles in $B(n,4n+1,2n)$ is negative. For $n=5$, the same holds for $B(n,4n+2,2n)$.
}
\end{rem}

\begin{prop}\label{p4.21}
Let $C$ be a non-hyperelliptic curve of genus $6$, $n$ and $s$ positive integers and $k\le 2n-s$. Then
$$
B(n,4n-s,k) \neq \emptyset
$$
in the following cases.
\begin{enumerate}
\item[(i)] $C$ general, $n\le 5$;
\item[(ii)] $\Cliff(C)=2$, $n\le5$, $\gcd(n,s)=1$;
\item[(iii)] $\Cliff(C) = 1$ or $C$ bielliptic, any $n$.
\end{enumerate} 
\end{prop}

\begin{proof}
The proof is analogous to the previous proof, using Proposition \ref{p2.5}(ii).
\end{proof}

\begin{rem} \label{r4.19}
{\rm
For $n = 4$, $5$, the Brill-Noether number of the bundles in $B(n,4n-1,2n-1)$ is negative.
}
\end{rem}

\begin{lem} \label{l3.12}
Let $C$ be a curve of genus $6$ with $\Cliff(C) = 2$. Let $L_i = K_C(-p_i)$ for $1 \leq i \leq r$, where $p_1, \cdots ,p_r$ are distinct
points of $C$. Then every proper subbundle $F$ of $E_{L_1} \oplus \cdots \oplus E_{L_r}$, which is not a partial direct sum of factors 
of $E_{L_1} \oplus \cdots \oplus E_{L_r}$, has
$$
d_F \leq \frac{9}{4}n_F -1.
$$
\end{lem}

\begin{proof}
The proof is by induction on $r$. If $r=1$, then $1 \leq n_F \leq 3$ and we require to show that $d_F < 2n_F$.  

By stability of $E_{L_1}$, $d_F \leq 2n_F$. If $d_F = 2n_F$, it is easy to see that $F$ is semistable and $E_{L_1}/F$ is stable. In view of Lemmas \ref{l3.7} and \ref{l4.8},
it follows that $h^0(E_{L_1}/F) \leq n_{E_{L_1}/F}$. Since also $h^0(F) \leq n_F$ by Proposition \ref{pln2}, this gives a contradiction.

Now suppose $r \geq 2$ and the lemma is proved for $r-1$ factors. Consider the projection $\pi: F \ra E_{L_1}$. We can assume without loss of generality that $\pi \neq 0$. 
If $\rk \; \pi = 4$, then by induction
$$
d_F \leq 9 + \frac{9}{4}(n_F -4) -1 = \frac{9}{4}n_F -1.
$$
If $\rk \; \pi = s < 4$, then 
$$
d_F \leq \frac{9}{4}s - 1 + \frac{9}{4}(n_F -s) = \frac{9}{4}n_F -1.
$$
This completes the proof.
\end{proof}

\begin{prop} \label{p3.13}
 Let $C$ be a curve of genus $6$ with $\Cliff(C) = 2$. Suppose $r \geq 1, \; p \in C$ and $L_1, \dots,L_r$ are as in Lemma \ref{l3.12}. Let
 $$
 0 \ra E_{L_1} \oplus \cdots \oplus E_{L_r} \ra E \ra \CC_p \ra 0
 $$
 be an extension classified by $(e_1, \dots,e_r)$, where the $e_i \in \Ext(\CC_p,E_{L_i})$ are all 
non-zero. Then $E$ is stable. Hence
 $$
 B(4r,9r+1,5r) \neq \emptyset.
 $$
\end{prop}

\begin{proof}
It follows from Lemma \ref{l3.12} that any proper subbundle $F$ of $E$ has $d_F \leq 
\frac{9}{4}n_F$. Hence $E$ is stable. 
\end{proof}

\begin{prop}  \label{p5.15}
Let $C$ be a curve of genus $6$ with $\Cliff(C) = 2$. Suppose $r \geq 1, \; p \in C$ and 
$L_1, \dots,L_r$ are as in Lemma \ref{l3.12}. Let
$$
0 \ra E \ra E_{L_1} \oplus \cdots \oplus E_{L_r} \ra \CC_p \ra 0
$$
be an elementary transformation such that all maps $E_{L_i} \ra \CC_p$ are non-zero. Then $E$ is stable. 
Hence
$$
B(4r,9r-1,5r-1) \neq \emptyset..
$$                    
\end{prop}

\begin{proof}
It follows from Lemma \ref{l3.12} that any proper subbundle $F$ of $E$ has 
$d_F \le \frac{9}{4}n_F -1$. So $E$ is stable.
\end{proof}

\section{$k=n+1$}\label{seck=n+1}

In this section, we investigate the Brill-Noether loci for $k=n+1$. We determine non-emptiness of $\widetilde{B}(n,d,n+1)$ in all cases and also that of $B(n,d,k)$ except in a small number of cases (at most $6$ for any value of $n$); the results are complete for $n=2$, $n=3$ and $n=5$.

\begin{prop} \label{p5.6}
Let $C$ be a non-hyperelliptic curve of genus $6$. Then
\begin{equation}\label{eq2d3}
\beta(2,d,3)\ge0\Leftrightarrow d\ge6.
\end{equation} 
Moreover,
\begin{enumerate}
\item[(i)] if $\Cliff(C) = 2$,
$$
\widetilde B(2,d,3) \neq \emptyset \; \Leftrightarrow \; B(2,d,3) \neq \emptyset \; \Leftrightarrow \; d \ge 6;
$$
\item[(ii)] if $C$ is trigonal,
$$
\widetilde{B}(2,d,3)\ne\emptyset\:\Leftrightarrow\;d\ge6,\quad B(2,d,3)\ne\emptyset\;\Leftrightarrow\;d\ge7;
$$
\item[(iii)] if $C$ is a smooth plane quintic,
$$
\widetilde B(2,d,3) \neq \emptyset \; \Leftrightarrow \; B(2,d,3) \neq \emptyset \; \Leftrightarrow \; 
 d \ge 5.
$$
\end{enumerate}
\end{prop}

\begin{proof}
\eqref{eq2d3} follows from \eqref{eqbeta}. Moreover, 
$\widetilde B(2,d,3) = \emptyset$ for $d \le 4$ by Proposition \ref{pln2}. To show that $B(2,7,3)$ and $B(2,8,3)$ are non-empty for $C$ trigonal, let $L$ be a line bundle of degree $3$ with $h^0(L)=1$ and apply Proposition \ref{p2.5}(i) to $L\oplus T$. The rest follows from Lemma 
\ref{l3.7}, Proposition \ref{p3.10} and Corollary \ref{c4.12} and tensoring by an effective line bundle.
\end{proof}

\begin{prop} \label{p5.8}
Let $C$ be a non-hyperelliptic curve of genus $6$. Then
\begin{equation}\label{eq3d4}
\beta(3,d,4)\ge0\Leftrightarrow d\ge8.
\end{equation} Moreover,
\begin{enumerate}
\item[(i)] if $C$ is either a smooth plane quintic or $\Cliff(C) =2$,
$$
\widetilde B(3,d,4) \neq \emptyset \; \Leftrightarrow \;  B(3,d,4) \neq \emptyset \; \Leftrightarrow \; d \geq 8;
$$
\item[(ii)] if $C$ is trigonal,
$$
\widetilde B(3,d,4) \neq \emptyset \; \Leftrightarrow B(3,d,4) \neq \emptyset \; \Leftrightarrow \; d \ge 7.
$$
\end{enumerate}
\end{prop}

\begin{proof} \eqref{eq3d4} follows from \eqref{eqbeta}. Moreover,
for any $C$, we have $\widetilde B(3,d,4) = \emptyset$ for $d \le 6$ by Proposition \ref{pln2}, also for $d\le7$ if $C$ is a smooth plane quintic or $\Cliff(C)=2$ by Lemma \ref{l4.8}. 

For general $C$, (i) now follows from \cite[Theorem 7.2]{bbn}. By semicontinuity, this implies the
equivalence 
$$
\widetilde B(3,d,4) \neq \emptyset \; \Leftrightarrow \;  d \geq 8
$$
in (i) and (again using Lemma \ref{l4.8}) a similar equivalence for $d\ge7$ in (ii). 

To complete the proof, we need to prove that $B(3,d,4)\ne\emptyset$ when $d$ is divisible by $3$, $d\ge9$. For this, it is sufficient to prove that $B(3,9,4)\ne\emptyset$.

Observe first that, when $C$ is a smooth plane quintic or $\Cliff(C)=2$, the general line bundle $N$ of degree $9$ has $h^0(N)=4$ (by Riemann-Roch) and is generated. For the latter, note that, if $N$ is not generated, then it is an elementary transformation of a line bundle in $B(1,8,4)$, which is isomorphic by Serre duality to $B(1,2,1)$ and therefore has dimension $2$. Our object now is to prove that, for the general such $N$, $E_N$ is stable. In fact, any quotient bundle $G$ of $E_N$ is generated and has $H^0(G^*)=0$. If $G$ is a line bundle, this implies that $d_G\ge4$ since $C$ is not trigonal; so $G$ does not contradict the stability of $E_N$. It remains therefore to consider the possibility that $G\in B(2,d,3)$ with $d_G\le6$. 

If $\Cliff(C)=2$, then $G=E_M$ for some $M\in B(1,6,3)$ by Propositions \ref{p3.10}(i) and \ref{p5.6}(i). The surjective homomorphism $E_N\to G$ gives rise to a non-zero homomorphism $M\to N$, so $N\simeq M\otimes L$ with $L\in B(1,3,1)$. Since $\dim B(1,6,3)\le1$ by Proposition \ref{p2.0}, the general $N$ does not have this property. On the other hand, if $C$ is a smooth plane quintic, we must have $G\simeq E_H$ by Proposition \ref{p3.10}(iii) and Lemma \ref{l3.7}. We then obtain $N\simeq H \otimes L$ with $L\in B(1,4,1)$, again contradicting the generality of $N$.

Now suppose $C$ is trigonal. In this case the proof that $E_N$ is stable is not valid and in fact, 
$E_N$ is always strictly semistable. 

Instead, let $G$ be the unique element of $B(3,7,4)$ (see Lemma    \ref{l4.8}).  We have $G$ generated and $h^0(G^*) =0$. Now consider elementary transformations 
$$
0 \ra E \ra G(p) \ra \CC_p \ra 0.
$$
Here $n_E =3, \; d_E = 9$ and $ h^0(E) \ge 4$, since $G \subset E$. It is clear that $E$ is semistable. 
If $E$ is strictly semistable, then using the stability of $G$ we see that there must be a diagram
\begin{equation} \label{e6.4}
\xymatrix{
0 \ar[r] & E \ar[d] \ar[r] & G(p) \ar[d] \ar[r] & \CC_p \ar[r] \ar@{=}[d] & 0\\
0  \ar[r] & F' \ar[r] \ar[d] & F(p)  \ar[r]  \ar[d] & \CC_p \ar[r] & 0,\\
& 0 & 0 &&
}
\end{equation}
with $n_{F'}\le2$ and $\mu(F') = 3$. 

If $n_{F'} =2$, then $F$ is a quotient of $G$ of rank $2$ and degree $5$. It follows that $F$ is generated, so $h^0(F)\ge3$. Since also $F$ is stable, Lemma \ref{l3.7} gives a contradiction.

The only alternative is that $n_{F'} = 1$. Then $F$ is a line bundle of degree $3$. Since $F$ is generated, It follows that $F \simeq T$. For any given homomorphism 
$G \ra T$ there is a one-dimensional space of homomorphisms $G(p) \ra \CC_p$ fitting into such a 
diagram. On the other hand, $G^* \otimes T$ is stable of slope $\frac{1}{3}$, So $h^0(G^* \otimes 
T) \le 2$ by Proposition \ref{pln2}. So up to scalar multiples there is at most a one-dimensional set of 
homomorphisms $G(p) \ra T(p)$. It follows that the general homomorphism $G(p) \ra \CC_p$ cannot 
fit into a diagram \eqref{e6.4}.  
This completes the proof of the proposition.
\end{proof}

\begin{rem}\label{remc}
{\rm When $\Cliff(C)=2$, we can obtain an explicit description of $B(3,8,4)$, namely
$$
B(3,8,4) = \{ E_M \;|\; M \in B(1,8,4) \}.
$$
The fact that $\{E_M \;|\; M \in B(1,8,4)\} \subset B(3,8,4)$ follows from \eqref{e3} and 
Proposition \ref{p2.1}(i).
Now suppose $E \in B(3,8,4)$. If $E$ is not generated, a negative elementary 
transformation yields a bundle $F$ with $n_F = 3,\;  d_F = 7$ and $h^0(F) =4$. By Lemma \ref{l4.8},
$F$ cannot be stable. The only possibility compatible with the stability of $E$ is that $F$ has a rank-2 stable subbundle $G$ of degree $5$. By Proposition \ref{p1.6}(ii), $h^0(G) \leq 2$, while $h^0(F/G) \leq 1$.
 This gives a contradiction. So $E$ is generated and hence 
$E = E_M$ for some $M \in B(1,8,4)$.}
\end{rem}

\begin{prop} \label{p5.9}
Let $C$ be a non-hyperelliptic curve of genus $6$. Then
\begin{equation}\label{eq4d5}
\widetilde B(4,d,5) \neq \emptyset \; \Leftrightarrow \; d \ge 9\Leftrightarrow\beta(4,d,5)\ge0.
\end{equation}
Moreover
\begin{enumerate}
\item[(i)] for general $C$, $B(4,d,5)\ne\emptyset$ if $d\ge9$;
\item[(ii)] for all $C$, $B(4,d,5)\ne\emptyset$ if $d\ge9$, $d\ne10,12,16$;
\item[(iii)] if $\Cliff(C)=2$ or $C$ is trigonal, $B(4,10,5)\ne\emptyset$;
\item[(iv)] if $C$ is trigonal, $B(4,16,5)\ne\emptyset$.
\end{enumerate} 
\end{prop}

\begin{proof} The second equivalence in \eqref{eq4d5} follows from \eqref{eqbeta}. Moreover,
for any $C$, $\widetilde 
B(4,d,5) = \emptyset$ for $d \le 8$ by Proposition \ref{pln2}. For general $C$, the fact that $B(4,d,5) \neq \emptyset$ for $d\ge9$ is contained in \cite[Theorem 7.3]{bbn}; this proves (i).  The first equivalence in \eqref{eq4d5} follows by semicontinuity, which also completes the proof of (ii) for $d$ odd.

If $\Cliff(C)=2$ or $C$ is trigonal, we have $B(2,5,3)=\emptyset$ by Lemma \ref{l3.7}; hence $B(4,10,5)=\widetilde{B}(4,10,5)\ne\emptyset$. If $\Cliff(C)=1$, then $B(4,14,5)\ne\emptyset$ by Proposition \ref{p4.21}. Proposition \ref{prop2.4} gives $B(4,16,5)\ne\emptyset$ for $C$ trigonal. If $\Cliff(C)=2$ or $C$ is a smooth plane quintic, then $B(4,20,5)\ne\emptyset$ by Proposition \ref{p2.11}. Tensoring by effective line bundles completes the proof.
\end{proof}

\begin{rem}\label{remb}
{\rm It is reasonable to conjecture that $B(4,d,5)\ne\emptyset$ in all the cases left open in Proposition \ref{p5.9}.}
\end{rem}

\begin{prop} \label{p6.9}
Let $C$ be a non-hyperelliptic curve of genus $6$. Then
\begin{equation}\label{eq5d6}
\widetilde B(5,d,6) \neq \emptyset \; \Leftrightarrow \; B(5,d,6) \neq \emptyset \; \Leftrightarrow \; d \ge 10\Leftrightarrow \beta(5,d,6)\ge0.
\end{equation}
\end{prop}

\begin{proof} The last equivalence in \eqref{eq5d6} follows from \eqref{eqbeta}. Moreover,
for $d \le 9$, $\widetilde B(5,d,6) = \emptyset$ by Proposition \ref{pln2}(i). 
For $d = 10$, we have $\widetilde B(5,10,6) = B(5,10,6)=\{D(K_C)\}$ by Proposition \ref{pln2}(iii). Tensoring by an effective line bundle shows that $B(5,d,6)\ne\emptyset$ whenever $d\ge10$ is divisible by $5$. 
For $d\ge11$, $B(5,d,6)\ne\emptyset$ for general $C$ by \cite[Corollary 6.3]{bbn2}. By semicontinuity, this holds for any $C$ except possibly when $d$ is divisible by $5$. This completes the proof.
\end{proof}

\begin{prop} \label{p6.8}
Let $C$ be a non-hyperelliptic curve of genus $6$ and suppose $n \ge 6$. Then 
\begin{enumerate}
 \item[(i)] $\widetilde B(n,d,n+1) \neq \emptyset \; \Leftrightarrow \; d \geq n+6\Leftrightarrow\beta(n,d,n+1)\ge0$;
 \item[(ii)] if $d \ge n+6$, then $B(n,d,n+1) \neq \emptyset$ except possibly when
 \begin{enumerate}
 \item[(a)] $n = 6, \; d = 14,15,16,20,21$ or $22$;
 \item[(b)] $n=8, \; d=18,20,26$ or $28$;
 \item[(c)] $n=9, \; d= 21$ or $30$;
 \item[(d)] $n=10, \; d= 22,24,25,32,34$ or $35$;
 \item[(e)] $n \ge 12,\; n$ even, $d = 2n+4$ or $ 3n+4$;
 \item[(f)] $n \ge 12, \; n$ divisible by $3, \; d = 2n+3$ or $3n+3$;
 \item[(g)] $n \ge 15, \; n$ divisible by $5,\; d = 2n+5$ or $3n+5$.
 \end{enumerate}
\end{enumerate}
\end{prop}

\begin{proof}
(i): The second equivalence follows from \eqref{eqbeta}. Moreover, if $d < n+6$, then $\widetilde B(n,d,n+1) = \emptyset$ by Proposition \ref{pln2}(i). If $d \geq n+6$, then $\widetilde B(n,d,n+1) \neq \emptyset$ for $C$ general by \cite[Theorem 5.1]{bbn2}.
By semicontinuity this holds for all $C$.

(ii): If $n+6 \le d \le 2n$, then $B(n,d,n+1) \neq \emptyset$ by Proposition \ref{pln2}(i) and (iii). If $d = 2n+1$, then $B(n,d,n+1) \neq \emptyset$ by Proposition \ref{p3.11}(iii). If $d = 2n+2$, the
same is true by Proposition \ref{p3.11}(ii) for $n \ge 11$. If $4n < d < 5n$, we can use Proposition \ref{p2.16}. The result follows by tensoring by effective line bundles,
except possibly for $d = 2n+3, 2n+4, 2n+5, 3n+3, 3n+4$ or $3n+5$ and, in addition, for $n < 11$, $d = 2n+2$ or $3n+2$. Since $B(n,d,n+1)$ is non-empty when $\gcd(n,d) = 1$ by (i), this leaves only the listed possibilities.
\end{proof}

\begin{rem}
{\rm 
If $\Cliff(C) = 1$ or $C$ is bielliptic, then, for $n\ge6$, 
$$
B(n,d,n+1) \neq \emptyset \quad \mbox{for} \quad d= 3n+3, 3n+4 \; \mbox{and} \; 3n+5
$$
by Proposition \ref{p2.17}(i). If $\Cliff(C)=2$ or $C$ is a smooth plane quintic, then $B(6,14,7)\ne\emptyset$ since $\widetilde{B}(6,14,7)\ne\emptyset$ and $B(3,7,4)=\emptyset$ by Lemma \ref{l4.8}; hence also $B(6,20,7)\ne\emptyset$. Note also that semicontinuity implies that $\dim\widetilde{B}(n,d,k)\ge\beta(n,d,k)$ whenever $B(n,d,k)\ne\emptyset$ on the general curve. It is therefore possible that one could show that $B(n,d,k)\ne\emptyset$ on a special curve by showing that the points corresponding to strictly semistable bundles form a subscheme of dimension $<\beta(n,d,k)$.
}
\end{rem}

\section{Bundles of ranks $2$ and $3$}\label{seclow}

In this section, we study the non-emptiness of Brill-Noether loci for $n=2$ and $n=3$. For $n=2$, we have a complete solution. For $n=3$ and $k=5,6$, we have a complete solution for $\widetilde{B}(3,d,k)$, but there are a few exceptions for $B(3,d,k)$ where we have not been able to obtain an answer.

\begin{prop}  \label{prk2}
Let $C$ be a non-hyperelliptic curve of genus $6$. Then
\begin{equation}\label{eq2d4}
\beta(2,d,4)\ge0\Leftrightarrow d\ge9.
\end{equation}
Moreover,
\begin{enumerate}
\item[(i)] $\widetilde B(2,d,4) = \emptyset$ for $d \le 5$;
\item[(ii)] $B(2,6,4) = \emptyset$; $\widetilde B(2,6,4) = \emptyset$ unless $C$ is trigonal, in 
which case $\widetilde B(2,6,4) \neq \emptyset$;
\item[(iii)] $B(2,7,4) = \emptyset$;
\item[(iv)] $\widetilde B(2,8,4) \neq \emptyset$ and $B(2,8,4) = \emptyset$;
\item[(v)] $B(2,d,4) \neq \emptyset$ for $d \ge 9$. 
\end{enumerate}
\end{prop}

\begin{proof}
\eqref{eq2d4} follows from \eqref{eqbeta}.

(i) follows from Proposition \ref{p2.2}(i). Note next that, if $E$ is a bundle of rank 2 and degree $d$
with $h^0 \ge 4$ and $E$ does not admit a line subbundle with $h^0 \ge 2$, then, by Lemma \ref{l2.2},
$d_E \ge d_4 = 9$. Except when $C$ is trigonal, this implies that $B(2,d,4) = \emptyset$ 
for $d \le 8$ and
that $\widetilde B(2,6,4) = \emptyset$. On the other hand, $\widetilde B(2,8,4) \neq \emptyset$ 
always, since $B(1,4,2) \neq \emptyset$.

So for (ii), (iii) and (iv) it remains to consider the case $C$ trigonal. In this case, $B(2,6,4) =
\emptyset$ by Remark \ref{r4.10}, while $ \widetilde B(2,6,4) \neq \emptyset$, since $B(1,3,2)
\neq \emptyset$. 

Now suppose $C$ is trigonal and $d = 7$. For any $E\in B(2,7,4)$, we must have an exact sequence
\begin{equation} \label{e7.1}
0 \ra T \ra E \ra M \ra 0
\end{equation} 
with $M\in B(1,4,2)$. Consider the multiplication map
$$
\mu:H^0(M) \otimes H^0(K_C \otimes T^*) \ra 
H^0(K_C \otimes M \otimes T^*).
$$
Now $h^0(M) = 2$, 
$h^0(K_C \otimes T^*) =4$ and $h^0(K_C \otimes M \otimes T^*) = 6$ by Riemann-Roch.
Moreover, by Remark \ref{reme}, either $M\simeq T(p)$ for some $p\in C$ or $M$ is generated and $M\simeq K_C\otimes T^{*2}$. Now consider the evaluation sequence for $M$,
$$0\ra L\ra H^0(M)\otimes \mathcal{O}_C\ra M'\ra 0.$$
Tensoring by $K_C\otimes T^*$, we see that $\operatorname{Ker}\mu\simeq H^0(L\otimes K_C\otimes T^*)$. If $M$ is generated, then $M'=M$ and $L\simeq M^*$; otherwise $M\simeq T(p)$, $M'\simeq T$ and $L\simeq T^*$. In the first case, $M\simeq K_C\otimes T^{*2}$ and $\operatorname{Ker}\mu\simeq H^0(T)$, in the second $\operatorname{Ker}\mu\simeq H^0(K_C\otimes T^{*2})$. In both cases $\dim(\operatorname{Ker}\mu)=2$, so $\mu$ is surjective. So, by Lemma \ref{llift}, there exists no non-trivial extension \eqref{e7.1} for which all 
sections of $M$ lift. This completes the proof of (iii).

For $C$ trigonal and $d = 8$, suppose $E\in B(2,8,4)$. By Proposition \ref{pln2}(i), $h^0(T^*\otimes E)\le1$. Hence, by Lemma \ref{l2.3}, 
$$
h^0(T\otimes E)\ge2h^0(E)-h^0(T^*\otimes E)\ge7.
$$ By Riemann-Roch,
$$h^0(K_C\otimes T^*\otimes E^*)\ge7+\chi(K_C\otimes T^*\otimes E^*)=3.$$
Thus $K_C\otimes T^*\otimes E^*\in B(2,6,3)$, which is empty by Proposition \ref{p3.10}(ii). This completes the proof of (iv).

For (v), $B(2,9,4) \neq \emptyset$ by Proposition \ref{p4.20}. If $\Cliff(C) = 2$ or $C$ is trigonal, then 
$B(2,10,4) \neq \emptyset$ by Proposition \ref{p4.18}. If $C$ is a smooth plane quintic, the same 
follows from Proposition \ref{p4.20}. So (v) follows by tensoring with an effective line bundle.
\end{proof}

\begin{prop}  \label{p7.2}
Let $C$ be a non-hyperelliptic curve of genus $6$. Then
\begin{equation}\label{eq2d5}
\beta(2,d,5)\ge0\Leftrightarrow d\ge11.
\end{equation}
Moreover,
\begin{enumerate}  
\item[(i)] $\widetilde B(2,d,5) = \emptyset$ for $d \le 8$;
\item[(ii)] $B(2,9,5)=\emptyset$;
\item[(iii)] $\widetilde{B}(2,10,5)\ne\emptyset$; $B(2,10,5)\ne\emptyset$ if ond only if $C$ is not a smooth plane quintic;
\item[(iv)] $B(2,d,5)\ne\emptyset$ for $d\ge11$.
\end{enumerate}
\end{prop}

\begin{proof}
\eqref{eq2d5} follows from \eqref{eqbeta}. 

By Lemma \ref{l2.2}, any bundle $E$ of rank 2 with $h^0(E) \ge 5$ and degree $d <d_6 = 12$
must admit a line subbundle $L$ with $h^0(L) \ge 2$. 

(i): If $d_L \ge 5$, this 
contradicts semistability for $d \le 8$. If $d_L = 4$, then $h^0(L) \le 2$ and $h^0(E/L) \le 2$. 
So $h^0(E) \le 4$, a contradiction. 
If $C$ is trigonal and $d_L = 3$, then $d_{E/L} \le 5$. So again $h^0(L) \le 2$ and $h^0(E/L) \le 2$, a contradiction.

(ii) For $\Cliff(C)=2$, the result follows from Proposition \ref{p1.6}(i). Next suppose $C$ is trigonal. If $d_L=4$, then $h^0(L)=2$ and $h^0(E/L)\le2$, so $h^0(E)\le4$. The only remaining possibility is that there is a exact sequence
$$
0\ra T\ra E\ra M\ra0,
$$
where $M\in B(1,6,3)$ and all sections of $M$ lift. We consider the multiplication map
$$
\mu:H^0(M)\otimes H^0(K_C\otimes T^*)\ra H^0(M\otimes K_C\otimes T^*).
$$
Here $h^0(M)=3$, $h^0(K_C\otimes T^*)=4$ and $h^0(M\otimes K_C\otimes T^*)=8$ by Riemann-Roch. Note that $M$ is generated, so $\operatorname{Ker}\mu=H^0(E_M^*\otimes K_C\otimes T^*)$. However, $E_M$ is not stable, but it is semistable by \cite[Proposition 4.9(d)]{ln}. So $E_M^*\otimes K_C\otimes T^*$ is a semistable bundle of rank $2$ and degree $8$. Hence, by (i), $h^0(E_M^*\otimes K_C\otimes T^*)\le4$. This implies that $\mu$ is surjective. Hence $B(2,9,5)=\emptyset$ by Lemma \ref{llift}.

If $C$ is a smooth plane quintic, the only possibility is that there is an exact sequence
$$
0\ra H(-p)\ra E\ra H\ra0,
$$
for which all sections of $H$ lift. The multiplication map
$$
\mu:H^0(H)\otimes H^0(K_C\otimes H^*(p))\ra H^0(K_C(p))
$$
is in fact surjective, so this is impossible by Lemma \ref{llift}. To see this, note that $K_C\simeq H^2$, so $\mu$ becomes $$\mu:H^0(H)\otimes H^0(H(p))\ra H^0(H^2(p)).$$ Since $H^0(H)=H^0(H(p))$ and $H^0(H^2)= H^0(H^2(p))$, it follows from the surjectivity of $H^0(H)\otimes H^0(H)\ra H^0(H^2)$ that $\mu$ is surjective.

(iii) is Proposition \ref{p4.18}.

(iv) $B(2,11,5)\simeq B(2,9,4)$ by Serre duality; similarly $B(2,12,5)\simeq B(2,8,3)$. The result follows from Propositions \ref{prk2}(v) and \ref{p5.6}.
 
\end{proof}

\begin{prop}\label{p7.3}
Let $C$ be a non-hyperelliptic curve of genus $6$ and $d\le10$. Then $\beta(2,d,k)<0$ for $k\ge6$ and
\begin{enumerate}
\item[(i)] $B(2,d,6)=\emptyset$; $\widetilde{B}(2,d,6)=\emptyset$ except that $\widetilde{B}(2,10,6)\ne\emptyset$ if $C$ is a smooth plane quintic;
\item[(ii)] $B(2,d,7)=\emptyset$.
\end{enumerate}
\end{prop}

\begin{proof}
The statement about $\beta(2,d,k)$ follows from \eqref{eqbeta}.

(i) follows from Propositions \ref{p1.6} and \ref{p4.11} and Remark \ref{rem4.19}, together with the fact that, on a smooth plane quintic, $H\oplus H$ is a semistable bundle of type $(2,10,6)$.

(ii) follows at once from Proposition \ref{p2.2}.
\end{proof}

The non-emptiness of all other Brill-Noether loci $B(2,d,k)$ can be determined from the above using Serre duality.\\

We turn now to the case $n=3$.

\begin{prop} \label{prk3}
Let $C$ be a non-hyperelliptic curve of genus $6$. Then
\begin{equation}\label{eq3d5}
\beta(3,d,5)\ge0\Leftrightarrow d\ge11.
\end{equation}
Moreover,
\begin{enumerate}
\item[(i)] $\widetilde B(3,d,5) = \emptyset$ for $d \le 8$;
\item[(ii)] $B(3,9,5) = \emptyset$;    $\widetilde B(3,9,5) \neq \emptyset \Leftrightarrow C$ trigonal;
\item[(iii)] $\widetilde{B}(3,d,5)\ne\emptyset$ for $d\ge10$; $B(3,d,5) \neq \emptyset$  for $d \ge10$, $d\ne12,15$.
\end{enumerate}
\end{prop}

\begin{proof}
\eqref{eq3d5} follows from \eqref{eqbeta}.

By Lemma \ref{l2.2}, any bundle $E$ of rank 3  with $h^0(E) \ge 5$ and degree $d_E < d_6 = 12$ 
has either a subpencil or a rank-2 subbundle $F$ with $h^0(F) \ge 3$.

For $d \le 7$, (i) follows from Propositions \ref{pln2} and \ref{p3.6}. Suppose $d=8$ and let 
$E \in B(3,8,5)$. Clearly $E$ has no subpencil. Let $F$ be a rank-2 subbundle with $h^0(F) \ge 3$.
It is easy to see that $F$ is stable with $d_F =5$. By Lemma \ref{l3.7}, $C$ is a smooth 
plane quintic and $F \simeq E_H$. But then $E/F \in B(1,3,2)$ which is empty.

(ii): if $E$ is a strictly semistable bundle of rank 3 and degree 9 with $h^0(E) = 5$, at least one factor in a 
Jordan-H\"older filtration must belong to $B(1,3,2)$, since $B(2,6,4) = \emptyset$ by Proposition 
\ref{prk2}(ii). This can only happen if $C$ is trigonal and clearly $\widetilde B(3,9,5) \neq 
\emptyset$ if $C$ is trigonal.

It remains to consider $E \in B(3,9,5)$. Certainly $E$ has no subpencil, so it must possess a 
subbundle $F$ of rank 2 with $h^0(F) \ge 3$. The only possibility is that $C$ is a smooth plane quintic 
and $F \simeq E_H$.  Moreover, $E/F \in B(1,4,2)$. Hence $E/F \simeq H(-p)$ for some $p \in C$ 
(see Remark \ref{reme}). So we have an exact sequence
$$
0 \ra E_H \ra E \ra H(-p) \ra 0
$$
and all sections of $H(-p)$ lift. Now consider the multiplication map
$$
\mu: H^0(H(-p)) \otimes H^0(K_C \otimes E_H^*) \ra H^0(K_C \otimes H \otimes E_H^*(-p)).
$$
Here $h^0(H(-p)) = 2, \; h^0(K_C \otimes E_H^*) = 8$ and $h^0(K_C \otimes H \otimes E_H^*(-p)) = 13$ by Riemann-Roch. Since $H(-p)$ is generated, $\Ker \mu \simeq H^0(K_C \otimes 
H^*(p) \otimes E_H^*)$. The bundle $K_C\otimes H^*(p)\otimes E_H^*$ is stable of rank 2 and degree 7. By Proposition \ref{prk2}(iii),
$\dim \Ker \mu \le 3$. Hence $\mu$ is surjective, a contradiction by Lemma \ref{llift}.

(iii): $B(3,10,5)$ and $B(3,11,5)$ are non-empty by Propositions \ref{p4.18} and \ref{p4.21}. It is clear that $\widetilde{B}(3,12,5)\ne\emptyset$ since $B(1,4,2)\ne\emptyset$. Moreover, by Serre duality, $B(3,18,5)\simeq B(3,12,2)\ne\emptyset$. The result follows by tensoring with effective line bundles.
\end{proof}

\begin{rem}\label{rbi}
{\rm If $C$ is bielliptic, $B(3,12,5)$ and $B(3,15,5)$ are non-empty by Proposition \ref{p1.7}. If $\Cliff(C)=1$ or  $C$ is general, $B(3,15,5)\ne\emptyset$ by Proposition \ref{p4.20}. There could be curves with $\Cliff(C)=2$ possessing fewer than 3 bundles in $B(1,4,2)$ and having $B(3,15,5)=\emptyset$.
}
\end{rem}

\begin{prop}
Let $C$ be a non-hyperelliptic curve of genus $6$. Then
\begin{equation}\label{eq3d6}
\beta(3,d,6)\ge0\Leftrightarrow d\ge14.
\end{equation}
Moreover,
\begin{enumerate}
\item[(i)] $\widetilde B(3,d,6) = \emptyset$ for $d \le 8$;
\item[(ii)] $B(3,9,6) = \emptyset; \quad \widetilde B(3,9,6) \neq \emptyset \; \Leftrightarrow C$ trigonal;
\item[(iii)] $B(3,d,6) = \emptyset$ for $d = 10, 11$;
\item[(iv)] $\widetilde{B}(3,d,6)\ne\emptyset$ for $d\ge12$; $B(3,d,6)\ne\emptyset$ for $d\ge13$, $d\ne15$.   
\end{enumerate}
\end{prop}

\begin{proof}
\eqref{eq3d6} follows from \eqref{eqbeta}.

By Lemma \ref{l2.2}, any bundle $E$ of rank 3  with $h^0(E) \ge 6$ and degree $d_E < d_9 = 15$ 
has either a subpencil or a rank-2 subbundle $F$ with $h^0(F) \ge 3$.

Note that $\widetilde B(3,9,6) \neq \emptyset$ if $C$ is trigonal. Part (i) and (ii) then follow from 
Proposition \ref{prk3}(i) and (ii). 

(iii): For $\Cliff(C) = 2$, the assertion follows from Proposition \ref{p1.6}(i). If $C$ is a smooth plane 
quintic and $E \in B(3,10,6)$, then a negative elementary transformation is semistable with 
$h^0 \ge 5$, contradicting Proposition \ref{prk3}(ii). Now suppose $E \in B(3,11,6)$.
Let $Q \in B(1,4,2)$. Then $K_C \otimes Q^* \otimes E^*$ is stable of rank 3 and degree 7.
By Lemma \ref{l4.8}, $h^0(K_C \otimes Q^* \otimes E^*) \le 3$. Hence 
$h^0(Q \otimes E) \leq 11$. On the other hand, since $Q$ is generated, we have 
$h^0(Q \otimes E) \ge 2 h^0(E) \ge 12$ by Lemma \ref{l2.3}, a contradiction.

Suppose now that $C$ is trigonal and $E \in B(3,d,6)$ for $d = 10$ or 11. 
If $E$ admits a subpencil, then $T \subset E$ and $E/T$ is semistable of degree 7 or 8 with
$h^0(E/T) \geq 4$. For $d_{E/T} = 7$, this contradicts Proposition \ref{prk2}(iii). For $d_{E/T} = 8,  E/T$ is 
strictly semistable, so has a line subbundle of degree 4. Pulling this back gives a stable rank-2 
subbundle $F$ of $E$ of degree 7. Now $E/F$ is a line bundle of degree 4, so $h^0(E/F)\le2$ and $F\in B(2,7,4)$. This contradicts
Proposition \ref{prk2}(iii). If $E$ does not admit a subpencil, then $E$ has a rank-$2$ subbundle $F$ with $h^0(F)\ge3$. By stability of $E$, $d_F\le7$. If $d_F=6$ or $7$, $F$ is necessarily semistable and $h^0(E/F)\le2$. So $h^0(F)\ge4$ and, by Proposition \ref{prk2}, the only possibility is $F\simeq T\oplus T$. If $d_F\le5$, then $F$ cannot be semistable by Proposition \ref{p5.6}(ii). It follows that $F$ must admit a subpencil. In either case, $E$ admits a subpencil, contradicting our assumption. This completes the proof of (iii).

(iv): $\widetilde{B}(3,12,6)\ne\emptyset$ since $B(1,4,2)\ne\emptyset$. $B(3,13,6)$ and $B(3,14,6)$ are non-empty by Proposition \ref{p4.20}. Moreover, $B(3,18,6)\simeq B(3,12,3)\ne\emptyset$. The result follows by tensoring with effective line bundles.
\end{proof}

\begin{rem}\label{r36}
{\rm If $C$ is general or bielliptic or $\Cliff(C)=1$, $B(3,15,6)\ne\emptyset$ by Proposition \ref{p4.20}.}
\end{rem}

\section{Extremal bundles}\label{secext}

In this section we investigate extremal bundles, by which we mean semistable bundles which maximize $\lambda= \frac{k}{n}$ for a given value of $\mu = \frac{d}{n}$. Among these bundles are those that lie on the upper bound lines that we have constructed, and it is these that we shall consider.

In the range $2 \le \mu \le \frac{9}{4}$, we know that the only stable bundles, for which the bound 
$\lambda =1 + \frac{1}{5}( \mu - 1)$ is attained, are $D(K_C)$ and the bundles $E_L$ with 
$L = K_C(-p)$ for some point $p \in C$ (Proposition \ref{p3.4}).
For $\mu > \frac{9}{4}$ we have less information.\\

For $C$ trigonal, we know that $B(3,7,4) \neq \emptyset$ by Lemma \ref{l4.8}.
We know of no other bundles attaining the upper bounds of subsection \ref{sub4.2} in the range $\mu > \frac{9}{4}$, except for $\mu = 3$.
In this case, the bound $\lambda=2$ is attained by $\oplus_{i=1}^n T$ for any $n$ and this is the only bundle that computes the Clifford index $\Cliff_n(C)$ \cite[Corollary 4.8]{ln3}.

\vspace{0.3cm}
For $C$ a smooth plane quintic, we know that $B(2,5,3) \neq \emptyset$ (Lemma \ref{l3.7}).
For $\frac52<\mu<3$, we know of no bundles which attain the bound $\lambda=\mu-1$. For $\mu \ge 3$, we have the following proposition (note that the conditions on $E$ say precisely that $E$ computes $\Cliff_n(C)$).

\begin{prop}
Let $C$ a smooth plane quintic and $E$ a semistable bundle of rank $n$ and degree $d$ such that $3n\le d\le5n$ and $h^0(E) 
= \frac{1}{2}(d + n)$. Then
$$
E \simeq \oplus_{i=1}^n H.
$$ 
\end{prop}

 \begin{proof}
It follows from Propositions \ref{pr4.12}  and  \ref{p4.13} that either $d = 3n$ or $d = 
5n$.

Suppose first that $d = 3n$. By Proposition \ref{prk2}(ii), we can assume that $n\ge3$. We have an exact sequence
$$
0 \ra E_H^* \otimes E \ra H^0(H) \otimes E \ra H \otimes E \ra 0,
$$
which gives
$$
h^0(E_H^* \otimes E) \ge 3h^0(E) - h^0(H \otimes E).
$$ 
Moreover, 
$$
h^0(K_C \otimes H^* \otimes E^*) \le \frac{6}{5}n
$$
by Proposition \ref{pln2} and, using this,
$$
h^0(H \otimes E) \le \frac{6}{5}n + \chi(H \otimes E) = \frac{21}{5}n.
$$
Since $h^0(E) = 2n$, this gives 
$$
h^0(E_H^* \otimes E) > 0.
$$
We therefore have a nonzero homomorphism $E_H \ra E$.
If $E$ is stable, this is an inclusion as a subbundle and we have an exact sequence
$$
0 \ra E_H \ra E \ra G \ra 0.
$$
Note that $3< \mu(G) \le 4$. If $G$ is not semistable, then a subbundle of $G$ contradicting 
semistability pulls back to a subbundle of $E$ contradicting stability. If $G$ is semistable, then 
$h^0(G) > \frac{1}{2}(d_G + n_G)$, which contradicts Proposition \ref{p2.2}(i).
Using Lemma \ref{lbb}, this shows that there is no semistable bundle $E$ of slope 3 with
$h^0(E) = \frac{1}{2}(d_E + n_E)$. 

It remains to consider the case $d = 5n$. By Remark \ref{rem4.19}, $E$ is a multiple extension of copies of $H$ with the property that at every step all sections of $H$ lift. It suffices to 
show that every such multiple extension splits. By induction, it suffices to show that, if 
$$
0 \ra H \ra F \ra H \ra 0
$$
is a non-split extension, then not all sections lift. This follows because
the canonical map 
$$
H^0(H) \otimes  H^0(H) = H^0(H)\otimes H^0(K_C \otimes H^*)  \ra  H^0(K_C) = H^0(H^2)
$$ 
is well known to be surjective.
\end{proof}
\

For $\Cliff(C)=2$, the situation is much clearer. We know that the bound $k= d-n$ is never attained for $\mu>\frac94$ (Proposition \ref{p3.7}). 
The bound for $3\le\mu\le5$ is now $\lambda\le\frac12\mu$ and this bound is attained for the values $3$, $\frac{10}3$, $4$ and $5$ of $\mu$ (in fact, 
$B(2,6,3)$, $B(3,10,5)$, $B(1,4,2)$ and $B(2,10,5)$ are non-empty). When $C$ is bielliptic, the bound is attained for all values of $\mu$. In fact, 
if we express the rational number $\lambda$ in its lowest terms as $\frac{k}{n}$, then $B(n,2k,k)\ne\emptyset$ by Proposition \ref{p1.7}.

For $C$ not bielliptic, we know no further bundles for which the bounds we have established can be attained, although there are many examples of bundles lying above the regions given by 
Corollary \ref{c2.2} and Proposition \ref{p2.11} (see the BN-map in Section \ref{secBN}).

\section{The BN-map} \label{secBN}

The following figures are the most significant part of the BN-map for non-hyperelliptic curves of genus 
6.
The map plots $\lambda = \frac{k}{n}$ against $\mu = \frac{d}{n}$.

We begin with the trigonal case (Figures 1 and 2). The solid lines indicate 
the upper bounds for non-emptiness of $B(n,d,k)$. In Figure 2, the upper solid line is the Clifford bound given by Proposition \ref{p2.2}(i), the lower ones the bounds given by Theorem \ref{t201}; note that these are only proved to be valid over part of the interval if $C$ does not admit a generated bundle in $B(1,4,2)$.

The shaded areas consist of 
points $(\mu, \lambda)$ for which there exist $(n,d,k)$ with 
$$
\frac{d}{n} = \mu,\quad  \frac{k}{n} =
 \lambda \quad \mbox{and} \quad B(n,d,k) \neq \emptyset.
 $$
 The black areas are given by Corollary \ref{c2.2} and Proposition  \ref{p2.17} 
and, for all $(n,d,k)$ corresponding to points in these areas, $B(n,d,k)$ is non-empty. 
The areas do not include the vertical line at $\mu = 4$ in Figure 2. 
In the large grey area, which 
corresponds to Proposition \ref{prop2.4} and does include the vertical line at $\mu=4$, there are some $(n,d,k)$ for which possibly 
$B(n,d,k) = \emptyset$. 
However, for any $(\mu,\lambda)$ in this area, there exist $(n,d,k)$ with $\mu = \frac{d}{n},\; \lambda = \frac{k}{n}$ such that $B(n,d,k) \neq \emptyset$.

The dots represent points for which some $B(n,d,k) \neq \emptyset$. The series of dots (and the small grey area close to $(\mu,\lambda)=(2,1)$)  
arise from Proposition \ref{p3.11} and Corollary \ref{c4.11}.
There are also isolated dots corresponding to Propositions \ref{p4.18}, \ref{p5.8} and \ref{p5.9}(iii) and Lemma \ref{l4.8}.

\begin{tikzpicture}
\path[draw][->] (2,-0.06) -- (2,8.1) node[pos=0.99,above] {$\lambda$} node[pos=0.99,left] {$2$};
\path[draw][->] (0,0 ) -- (10.2,0) node[pos=1.02,right] {$\mu$} node[pos=0.21,below] {$2$} node[pos=0.98,below] {$3$} node[pos=0,left] {$1$}; 
               
\path[draw] (10, 0) -- (10,8)  node[pos=0.16,right] {$7/6$} 

node[pos=0.4,right] {}; 
\path[draw][dashed] (2,8) -- (10,8)  node[pos=1,right] {$T$};
\path[draw,line width=1.2pt] (4,2) -- (6.31,4.31); 
\path[draw,line width=1.2pt] (2,1.6) -- (4,2); 
\path[draw,line width=1.2pt] (6.31,4.31) -- (10,5.33) node[pos=1,right] {$5/3$} ;
\path[fill=black] (2,0)--(10,0)--(10,1.33)--cycle;
\path[fill=black] (0,0)--(2,0)--(2,1.33)--(0,1.05)-- cycle;
\fill (4,2) circle (2pt);
\fill (10,1.6) circle (2pt);
\fill (2,1.6) circle (2pt);
\fill (6,2) circle (2pt);
\path[draw][dashed] (2,4) -- (10,4) node[pos=1,right] {$3/2$};
\path[draw][dashed] (4.67,0) -- (4.67,8) node[pos=0,below]{$7/3$};
\path[draw][dashed] (4,0) -- (4,8)  node[pos=0.33,below]{$E_L$} node[pos=0,below]{$9/4$};
\path[draw][dashed] (6.31,0) -- (6.31,8) node[pos=0,below]{$33/13$};
\path[draw][dashed] (2,1.33) -- (10,1.33);
\path[draw][dashed] (2,1.6) -- (10,1.6) node[pos=0,left]{$D(K_C)$} node[pos=1,right]{$6/5$};
\path[draw][dashed] (2,2) -- (10,2) node[pos=1,right]{$5/4$};
\fill (10,8) circle (2pt);

\fill (4.67,2.67) circle (2pt);

\fill (3.6,1.6) circle (1pt);
\fill (3.33,1.33) circle (1pt);
\fill (3.14,1.14) circle (1pt);
\fill (3,1) circle (1pt);
\fill (2.89,0.89) circle (1pt);
\fill (2.8,0.8) circle (1pt);
\fill (2.73,0.73) circle (1pt);
\fill (2.67,0.67) circle (1pt);
\fill (2.62,0.62) circle (1pt);
\fill (2.57,0.57) circle (1pt);
\fill (2.53,0.53) circle (1pt);
\fill (2.5,0.5) circle (1pt);
\path[draw,line width=1.3pt] (2,0) -- (2.5,0.5);

\fill (2.8,1.6) circle (1pt);
\fill (2.73,1.45) circle (1pt);
\fill (2.67,1.33) circle (1pt);
\fill (2.62,1.23) circle (1pt);
\fill (2.57,1.14) circle (1pt);
\fill (2.53,1.07) circle (1pt);
\fill (2.5,1) circle (1pt);
\fill (2.47,0.94) circle (1pt);
\fill (2.44,0.89) circle (1pt);
\fill (2.42,0.84) circle (1pt);
\fill (2.4,0.8) circle (1pt);
\path[draw,line width=1.3pt] (2,0) -- (2.4,0.8);

\fill (2.53,1.6) circle (1pt);
\fill (2.5,1.5) circle (1pt);
\fill (2.47,1.41) circle (1pt);
\fill (2.44,1.33) circle (1pt);
\fill (2.42,1.26) circle (1pt);
\fill (2.4,1.2) circle (1pt);
\fill (2.38,1.14) circle (1pt);
\fill (2.36,1.1) circle (1pt);
\fill (2.35,1.04) circle (1pt);
\fill (2.33,1) circle (1pt);
\fill (2.32,0.96) circle (1pt);
\path[draw,line width=1.3pt] (2,0) -- (2.32,0.96);

\fill (2.4,1.6) circle (1pt);
\fill (2.38,1.52) circle (1pt);
\fill (2.36,1.45) circle (1pt);
\fill (2.35,1.39) circle (1pt);
\fill (2.33,1.33) circle (1pt);
\fill (2.32,1.28) circle (1pt);
\path[draw,line width=1.3pt] (2,0) -- (2.32,1.28);

\fill (2.32,1.6) circle (1pt);
\fill (2.31,1.54) circle (1pt);
\fill (2.30,1.48) circle (1pt);
\fill (2.29,1.43) circle (1pt);
\fill (2.28,1.38) circle (1pt);
\path[draw,line width=1.3pt] (2,0) -- (2.28,1.38);

\fill (3.45,0.73) circle (1pt);
\fill (3.33,0.67) circle (1pt);
\fill (3.23,0.62) circle (1pt);
\fill (3.14,0.57) circle (1pt);
\fill (3.07,0.53) circle (1pt);
\fill (3,0.5) circle (1pt);
\fill (2.94,0.47) circle (1pt);
\fill (2.89,0.44) circle (1pt);
\fill (2.84,0.42) circle (1pt);
\fill (2.8,0.4) circle (1pt);
\fill (2.76,0.38) circle (1pt);
\path[draw,line width=1.3pt] (2,0) -- (2.76,0.38);

\fill (2.36,0.76) circle (1pt);
\fill (2.35,0.73) circle (1pt);
\path[draw,line width=1.3pt] (2,0) -- (2.35,0.73);

\fill (2.52,0.77) circle (1pt);
\fill (2.5,0.75) circle (1pt);
\path[draw,line width=1.6pt] (2,0) -- (2.48,0.73);

\fill (2.31,0.78) circle (1pt);
\path[draw,line width=1.3pt] (2,0) -- (2.31,0.78);

\fill (2.216,1.3) circle (1pt);
\fill (2.21,1.26) circle (1pt);
\path[draw,line width=1.3pt] (2,0) -- (2.205,1.23);
\fill (2.186,1.3) circle (1pt);
\fill (2.18,1.27) circle (1pt);
\path[shade,draw] (2,0)--(2.22,1.33)--(2,1.33)-- cycle;

\fill (8,2) circle (1pt);
\fill (8.67,2.67) circle (1pt);
\fill (9,3) circle (1pt);
\fill (9.2,3.2) circle (1pt);
\fill (9.33,3.33) circle (1pt);
\fill (9.43,3.43) circle (1pt);
\fill (9.5,3.5) circle (1pt);
\fill (9.56,3.56) circle (1pt);
\fill (9.6,3.6) circle (1pt);
\fill (9.63,3.63) circle (1pt);
\path[draw,line width=1.3pt] (9.63,3.63) -- (10,4);

\fill (7.33,2.67) circle (1pt);

\fill (5.2,1.6) circle (1pt);

\fill (8.4,1.6) circle (1pt);
\fill (6.8,1.6) circle (1pt);

\fill (10,2.66) circle (2pt);

\fill (8.67,1.33) circle (1pt);

\path[draw,line width=0pt] (2.1,-1.3) -- (7.9,-1.3) node[pos=0.5,sloped,above] {Figure 1: $C$ trigonal, $2 \leq \mu \leq 3$};

\end{tikzpicture}

\begin{tikzpicture}
\path[draw][->] (0,-0.06) -- (0,10.1) node[pos=0.99,above] {$\lambda$} node[pos=0.99,left] {$3$} node[pos=0.33,left] {$5/3$};
\path[draw][->] (0,0 ) -- (10.2,0) node[pos=1.02,right] {$\mu$} node[pos=0,below] {$3$} node[pos=0.98,below] {$5$} node[pos=0,left] {$1$}; 
\draw[line width=1pt] (0,2.24) .. controls(5,3.9) .. (10,6.18);              
\path[draw] (10, 0) -- (10,10); 

node[pos=0.4,right] {}; 
\path[draw][dashed] (0,7.5) -- (10,7.5);
\path[draw,line width=1.2pt] (0,5) -- (10,10); 
\path[fill=black] (5,0)--(10,0)--(10,5)--(5,5) -- cycle;
\path[fill=black] (0,0)--(5,0)--(5,5)-- cycle;
\path[shade,draw] (0,0) --(0,3.33)--(5,5) -- cycle;
\path[shade,draw] (5,5) --(10,6.67)--(10,5) -- cycle;
\path[draw,line width=1.2pt] (0,4.17) -- (2.31,5) node[pos=0,left] {$11/6$};
\path[draw][dashed] (2.31,0) -- (2.31,10) node[pos=0,below] {$45/13$}  ;
\path[draw][dashed] (0,6.67) -- (10,6.67) node[pos=1,right]{$7/3$};
\path[draw,line width=1.2pt] (2.31,5) -- (3.75,5);
\path[draw,line width=1.2pt] (3.75,5) -- (5,5.5);
\path[draw,line width=1.2pt] (5,6.67) -- (7.31,7.69);
\path[draw,line width=1.2pt] (7.31,7.69) -- (8.75,7.81);
\path[draw,line width=1.2pt] (8.75,7.81) -- (10,8.41);
\path[draw][dashed] (8.75,0) -- (8.75,10) node[pos=0,below]{$19/4$};
\path[draw][dashed] (7.31,0) -- (7.31,10) node[pos=0,below]{$58/13$};
\path[draw][dashed] (3.75,0) -- (3.75,10) node[pos=0,below] {$15/4$};
\path[draw] (5,0) -- (5,10)  node[pos=0,below]{$4$};
\path[draw][dashed] (0,5) -- (10,5) node[pos=1,right] {$2$}  ;
\path[draw][dashed] (0,10) -- (10,10) node[pos=0,left]{$3$};
\path[draw][dashed] (0,5.5) -- (10,5.5) node[pos=0,left]{$21/10$};
\path[draw][dashed] (0,7.69) -- (10,7.69) node[pos=0,left]{$33/13$};
\path[draw][dashed] (0,7.81) -- (10,7.81) node[pos=1,right]{$41/16$};
\path[draw][dashed] (0,8.42) -- (10,8.42) node[pos=1,right]{$161/60$};
\fill (10,7.5) circle (2pt) node[right]{$5/2$};
draw[line width=1pt] (0,2.24) .. controls(5,3.9) .. (10,6.18);

\path[draw,line width=0pt] (2.1,-1.3) -- (7.9,-1.3) node[pos=0.5,sloped,above] {Figure 2: $C$ trigonal, $3 \leq \mu \leq 5$};

\end{tikzpicture}

\vspace{1cm}

When $C$ is a smooth plane quintic (see Figures 3 and 4), the solid black lines represent the upper bounds established in Theorem \ref{t202}. For $3\le\mu\le5$, the upper solid line is again the Clifford bound. The black areas are exactly as in the trigonal case, but the large grey area is not present (in general, we do not know about stability for points in this area). In its place, there is a new series of dots arising from Corollary \ref{c4.12} and a new isolated point at $(\frac{10}3,\frac53)$ (Proposition \ref{p4.18}). There are also further isolated dots at $(\frac52,\frac32)$ (Lemma \ref{l3.7}) and $(3,\frac32)$ (Proposition \ref{p3.10}).

\begin{tikzpicture}
\path[draw][->] (2,-0.06) -- (2,8.1) node[pos=0.99,above] {$\lambda$} node[pos=0.99,left] {$2$};
\path[draw][->] (0,0 ) -- (10.2,0) node[pos=1.02,right] {$\mu$} node[pos=0.21,below] {$2$} node[pos=0.98,below] {$3$} node[pos=0,left] {$1$}; 
               
\path[draw] (10, 0) -- (10,8)  node[pos=0.16,right] {$7/6$} 

node[pos=0.4,right] {}; 
\path[draw][dashed] (2,8) -- (10,8);
\path[draw,line width=1.2pt] (4,2) -- (4.67,2.67); 
\path[draw,line width=1.2pt] (2,1.6) -- (4,2); 
\path[draw,line width=1.2pt] (4.67,2.67) -- (6,2.67) ;
\path[draw,line width=1.2pt] (6,4) -- (10,8) node[pos=0,below]{$E_H$};
\path[fill=black] (2,0)--(10,0)--(10,1.33)--cycle;
\path[fill=black] (0,0)--(2,0)--(2,1.33)--(0,1.05)-- cycle;
\fill (4,2) circle (2pt);
\fill (10,1.6) circle (2pt);
\fill (10,4) circle (2pt);
\fill (2,1.6) circle (2pt);
\path[draw][dashed] (2,4) -- (10,4) node[pos=1,right] {$3/2$};
\path[draw][dashed] (4.67,0) -- (4.67,8) node[pos=0,below]{$7/3$};
\path[draw][dashed] (4,0) -- (4,8)  node[pos=0.33,below]{$E_L$} node[pos=0,below]{$9/4$};
\path[draw][dashed] (2,1.33) -- (10,1.33);
\path[draw][dashed] (2,1.6) -- (10,1.6) node[pos=0,left]{$D(K_C)$} node[pos=1,right]{$6/5$};
\path[draw][dashed] (2,2) -- (10,2) node[pos=1,right]{$5/4$};
\fill (6,4) circle (2pt);

\fill (3.6,1.6) circle (1pt);
\fill (3.33,1.33) circle (1pt);
\fill (3.14,1.14) circle (1pt);
\fill (3,1) circle (1pt);
\fill (2.89,0.89) circle (1pt);
\fill (2.8,0.8) circle (1pt);
\fill (2.73,0.73) circle (1pt);
\fill (2.67,0.67) circle (1pt);
\fill (2.62,0.62) circle (1pt);
\fill (2.57,0.57) circle (1pt);
\fill (2.53,0.53) circle (1pt);
\fill (2.5,0.5) circle (1pt);
\path[draw,line width=1.3pt] (2,0) -- (2.5,0.5);

\fill (2.8,1.6) circle (1pt);
\fill (2.73,1.45) circle (1pt);
\fill (2.67,1.33) circle (1pt);
\fill (2.62,1.23) circle (1pt);
\fill (2.57,1.14) circle (1pt);
\fill (2.53,1.07) circle (1pt);
\fill (2.5,1) circle (1pt);
\fill (2.47,0.94) circle (1pt);
\fill (2.44,0.89) circle (1pt);
\fill (2.42,0.84) circle (1pt);
\fill (2.4,0.8) circle (1pt);
\path[draw,line width=1.3pt] (2,0) -- (2.4,0.8);

\fill (2.53,1.6) circle (1pt);
\fill (2.5,1.5) circle (1pt);
\fill (2.47,1.41) circle (1pt);
\fill (2.44,1.33) circle (1pt);
\fill (2.42,1.26) circle (1pt);
\fill (2.4,1.2) circle (1pt);
\fill (2.38,1.14) circle (1pt);
\fill (2.36,1.1) circle (1pt);
\fill (2.35,1.04) circle (1pt);
\fill (2.33,1) circle (1pt);
\fill (2.32,0.96) circle (1pt);
\path[draw,line width=1.3pt] (2,0) -- (2.32,0.96);

\fill (2.4,1.6) circle (1pt);
\fill (2.38,1.52) circle (1pt);
\fill (2.36,1.45) circle (1pt);
\fill (2.35,1.39) circle (1pt);
\fill (2.33,1.33) circle (1pt);
\fill (2.32,1.28) circle (1pt);
\path[draw,line width=1.3pt] (2,0) -- (2.32,1.28);

\fill (2.32,1.6) circle (1pt);
\fill (2.31,1.54) circle (1pt);
\fill (2.30,1.48) circle (1pt);
\fill (2.29,1.43) circle (1pt);
\fill (2.28,1.38) circle (1pt);
\path[draw,line width=1.3pt] (2,0) -- (2.28,1.38);

\fill (3.45,0.73) circle (1pt);
\fill (3.33,0.67) circle (1pt);
\fill (3.23,0.62) circle (1pt);
\fill (3.14,0.57) circle (1pt);
\fill (3.07,0.53) circle (1pt);
\fill (3,0.5) circle (1pt);
\fill (2.94,0.47) circle (1pt);
\fill (2.89,0.44) circle (1pt);
\fill (2.84,0.42) circle (1pt);
\fill (2.8,0.4) circle (1pt);
\fill (2.76,0.38) circle (1pt);
\path[draw,line width=1.3pt] (2,0) -- (2.76,0.38);

\fill (2.36,0.76) circle (1pt);
\fill (2.35,0.73) circle (1pt);
\path[draw,line width=1.3pt] (2,0) -- (2.35,0.73);

\fill (2.52,0.77) circle (1pt);
\fill (2.5,0.75) circle (1pt);
\path[draw,line width=1.6pt] (2,0) -- (2.48,0.73);

\fill (2.31,0.78) circle (1pt);
\path[draw,line width=1.3pt] (2,0) -- (2.31,0.78);

\fill (2.216,1.3) circle (1pt);
\fill (2.21,1.26) circle (1pt);
\path[draw,line width=1.3pt] (2,0) -- (2.205,1.23);

\fill (2.186,1.3) circle (1pt);
\fill (2.18,1.27) circle (1pt);
\path[shade,draw] (2,0)--(2.22,1.33)--(2,1.33)-- cycle;

\fill (8,2) circle (1pt);
\fill (8.67,2.67) circle (1pt);
\fill (9,3) circle (1pt);
\fill (9.2,3.2) circle (1pt);
\fill (9.33,3.33) circle (1pt);
\fill (9.43,3.43) circle (1pt);
\fill (9.5,3.5) circle (1pt);
\fill (9.56,3.56) circle (1pt);
\fill (9.6,3.6) circle (1pt);
\fill (9.63,3.63) circle (1pt);
\path[draw,line width=1.3pt] (9.63,3.63) -- (10,4);

\fill (7.33,2.67) circle (1pt);

\fill (5.2,1.6) circle (1pt);

\fill (8.4,1.6) circle (1pt);
\fill (6.8,1.6) circle (1pt);

\fill (10,2.66) circle (2pt);

\fill (8.67,1.33) circle (1pt);

\path[draw,line width=0.2pt] (2.1,-1.3) -- (7.9,-1.3) node[pos=0.5,sloped,above] {Figure 3: $C$ smooth plane quintic, $2 \leq \mu \leq 3$};
\end{tikzpicture} 

\begin{tikzpicture}
\path[draw][->] (0,-0.06) -- (0,10.1) node[pos=0.99,above] {$\lambda$} node[pos=0.99,left] {$3$};
\path[draw][->] (0,0 ) -- (10.2,0) node[pos=1.02,right] {$\mu$} node[pos=0,below] {$3$} node[pos=0.98,below] {$5$} node[pos=0,left] {$1$}; 
\path[draw] (10, 0) -- (10,10); 

node[pos=0.4,right] {}; 
\path[draw,line width=1.2pt] (0,5) -- (10,10); 
\path[fill=black] (5,0)--(10,0)--(10,5)--(5,5) -- cycle;
\path[fill=black] (0,0)--(5,0)--(5,5)-- cycle;
\path[draw,line width=1.2pt] (0,5) -- (2.5,5); 
\path[draw,line width=1.2pt] (2.5,4.58) -- (3.33,5);
\path[draw,line width=1.2pt] (3.75,5) -- (5,5.5);
\path[draw,line width=1.2pt] (5,6.67) -- (7.5,8.33);
\path[draw,line width=1.2pt] (7.5,8.33) -- (10,8.33);
\path[draw,line width=1.2pt] (3.33,5) -- (3.75,5);
\path[draw][dashed] (1.67,0) -- (1.67,10) node[pos=0,below] {$10/3$}  ;
\path[draw][dashed] (0,6.67) -- (10,6.67) node[pos=0,left] {$7/3$} ;
\path[draw][dashed] (7.5,0) -- (7.5,10) node[pos=0,below]{$9/2$};
\path[draw][dashed] (2.5,0) -- (2.5,10) node[pos=0,below]{$7/2$};

\path[draw][dashed] (3.33,0) -- (3.33,10) node[pos=1,above]{$11/3$};
\path[draw][dashed] (3.75,0) -- (3.75,10) node[pos=0,below] {$15/4$};
\path[draw] (5,0) -- (5,10)  node[pos=0,below]{$4$};
\path[draw][dashed] (0,5) -- (10,5) node[pos=1,right] {$2$}  ;
\path[draw][dashed] (0,10) -- (10,10) node[pos=0,left]{$3$};
\path[draw][dashed] (0,5.5) -- (10,5.5) node[pos=0,left]{$21/10$};
\path[draw][dashed] (0,4.58) -- (10,4.58) node[pos=0,left]{$23/12$};
\path[draw][dashed] (0,8.33) -- (10,8.33) node[pos=1,right]{$8/3$};

\fill (1.67,3.33) circle (2pt) node[right]{$(\frac{10}{3},\frac{5}{3})$};
draw[line width=1pt] (0,2.24) .. controls(5,3.9) .. (10,6.18);

\fill (0.5,1.0) circle (1.2pt);
\fill (0.45,0.9) circle (1.2pt);
\fill (0.42,0.84) circle (1.2pt);
\fill (0.38,0.76) circle (1.2pt);
\fill (0.36,0.72) circle (1.2pt);
\fill (0.33,0.67) circle (1.2pt);
\fill (0.31,0.63) circle (1.2pt);
\fill (0.29,0.58) circle (1.2pt);  
\path[draw,line width=1.4pt] (0,0) -- (0.28,0.56);

\fill (0.33,1) circle (1.2pt);
\fill (0.31,0.94) circle (1.2pt);
\fill (0.29,0.88) circle (1.2pt);
\fill (0.28,0.83) circle (1.2pt);
\fill (0.26,0.79) circle (1.2pt); 
\path[draw,line width=1.4pt] (0,0) -- (0.26,0.79);

\fill (0.25,1) circle (1.2pt);
\fill (0.24,0.95) circle (1.2pt);
\fill (0.23,0.91) circle (1.2pt);
\fill (0.22,0.87) circle (1.2pt);
\path[draw,line width=1.4pt] (0,0) -- (0.22,0.87);

\fill (0.2,1) circle (1.2pt);
\fill (0.19,0.96) circle (1.2pt);
\fill (0.18,0.93) circle (1.2pt);
\fill (0.18,0.89) circle (1.2pt);
\path[draw,line width=1.4pt] (0,0) -- (0.18,0.89);

\fill (0.32,0.48) circle (1.2pt);
\fill (0.31,0.47) circle (1.2pt);
\path[draw,line width=1.6pt] (0,0) -- (0.30,0.45);

\path[shade,draw] (0,0)--(0.13,0.83)--(0,0.83)-- cycle;

\fill (1.25,2.5) circle (1.2pt);
\fill (0.83,2.5) circle (1.2pt);
\fill (0.63,2.5) circle (1.2pt);
\fill (0.5,2.5) circle (1.2pt);
\fill (0.42,2.5) circle (1.2pt);
\fill (0.36,2.5) circle (1.2pt);
\fill (0.31,2.5) circle (1.2pt);
\fill (0.28,2.5) circle (1.2pt);
\fill (0.25,2.5) circle (1.2pt);
\path[draw,dashed] (0,2.5) -- (5,2.5) node[pos=0,left] {$3/2$};
\path[draw,line width=1.6pt] (0,2.5) -- (0.25,2.5);

\fill (0,1.66) circle (2pt);

\fill(10,10) circle(2pt);

\path[draw,line width=0pt] (2.1,-1.3) -- (7.9,-1.3) node[pos=0.5,sloped,above] {Figure 4: $C$ smooth plane quintic, $3 \leq \mu \leq 5$};

\end{tikzpicture}

We turn finally to the case where $\Cliff(C)=2$ (Figures 5 and 6). Again the solid black lines represent the bounds established in Theorem \ref{th3.18}. For $3\le\mu\le5$, this is the Clifford bound given by Proposition \ref{p1.6}(i) and is the precise bound when $C$ is bielliptic. The BN-curve (the thin curve in the figures) is given by $\lambda(\lambda - \mu  + 5) = 5$ (or $\beta(n,d,k) = 1$). In the interval $3\le\mu\le5$, even in the non-bielliptic case, there are several points which lie above this curve and some of these correspond to bundles for which $\beta(n,d,k)<0$.

In the interval $2\le\mu\le3$, the black areas are given by Corollary \ref{c2.2} as usual. The series of dots are given by Proposition \ref{p3.11}, Corollaries \ref{c4.11} and \ref{c4.12} and Propositions \ref{p3.13} and \ref{p5.15} (the last two are  established only for $\Cliff(C)=2$). In the interval $3\le\mu\le5$, existence results are only relevant for non-bielliptic curves. The black areas are given by Corollary \ref{c2.2} and Proposition \ref{p2.16} and the large grey area by Proposition \ref{p2.11}. There are some additional series of dots given by Propositions \ref{p4.20} and \ref{p4.21}.

\begin{tikzpicture}
\path[draw][->] (2,-0.06) -- (2,8.1) node[pos=0.99,above] {$\lambda$} node[pos=0.99,left] {$2$};
\path[draw][->] (0,0 ) -- (10.2,0) node[pos=1.02,right] {$\mu$} node[pos=0.21,below] {$2$} node[pos=0.98,below] {$3$} node[pos=0,left] {$1$}; 
               
\path[draw] (10, 0) -- (10,8)  node[pos=0.16,right] {$7/6$} 

node[pos=0.4,right] {}; 
\path[draw][dashed] (2,8) -- (10,8);
\path[draw,line width=1.2pt] (4.67,2.67) -- (10,4); 
\path[draw,line width=1.2pt] (2,1.6) -- (4,2); 
\path[draw,line width=1.2pt] (4,2) -- (4.67,2.67);
\path[fill=black] (2,0)--(10,0)--(10,1.33)--cycle;
\path[fill=black] (0,0)--(2,0)--(2,1.33)--(0,1.05)-- cycle;
\fill (4,2) circle (2pt);
\fill (10,1.6) circle (2pt);
\fill (10,4) circle (2pt);
\fill (2,1.6) circle (2pt);
\path[draw][dashed] (2,4) -- (10,4);
\path[draw][dashed] (4.67,0) -- (4.67,8) node[pos=0,below]{$7/3$};
\path[draw][dashed] (4,0) -- (4,8)  node[pos=0.33,below]{$E_L$} node[pos=0,below]{$9/4$};
\path[draw][dashed] (2,1.33) -- (10,1.33);
\path[draw][dashed] (2,1.6) -- (10,1.6) node[pos=0,left]{$D(K_C)$} node[pos=1,right]{$6/5$};
\path[draw][dashed] (2,2) -- (10,2) node[pos=1,right]{$5/4$};
\fill (3.6,1.6) circle (1pt);
\fill (3.33,1.33) circle (1pt);
\fill (3.14,1.14) circle (1pt);
\fill (3,1) circle (1pt);
\fill (2.89,0.89) circle (1pt);
\fill (2.8,0.8) circle (1pt);
\fill (2.73,0.73) circle (1pt);
\fill (2.67,0.67) circle (1pt);
\fill (2.62,0.62) circle (1pt);
\fill (2.57,0.57) circle (1pt);
\fill (2.53,0.53) circle (1pt);
\fill (2.5,0.5) circle (1pt);
\path[draw,line width=1.3pt] (2,0) -- (2.5,0.5);

\fill (2.8,1.6) circle (1pt);
\fill (2.73,1.45) circle (1pt);
\fill (2.67,1.33) circle (1pt);
\fill (2.62,1.23) circle (1pt);
\fill (2.57,1.14) circle (1pt);
\fill (2.53,1.07) circle (1pt);
\fill (2.5,1) circle (1pt);
\fill (2.47,0.94) circle (1pt);
\fill (2.44,0.89) circle (1pt);
\fill (2.42,0.84) circle (1pt);
\fill (2.4,0.8) circle (1pt);
\path[draw,line width=1.3pt] (2,0) -- (2.4,0.8);

\fill (2.53,1.6) circle (1pt);
\fill (2.5,1.5) circle (1pt);
\fill (2.47,1.41) circle (1pt);
\fill (2.44,1.33) circle (1pt);
\fill (2.42,1.26) circle (1pt);
\fill (2.4,1.2) circle (1pt);
\fill (2.38,1.14) circle (1pt);
\fill (2.36,1.1) circle (1pt);
\fill (2.35,1.04) circle (1pt);
\fill (2.33,1) circle (1pt);
\fill (2.32,0.96) circle (1pt);
\path[draw,line width=1.3pt] (2,0) -- (2.32,0.96);

\fill (2.4,1.6) circle (1pt);
\fill (2.38,1.52) circle (1pt);
\fill (2.36,1.45) circle (1pt);
\fill (2.35,1.39) circle (1pt);
\fill (2.33,1.33) circle (1pt);
\fill (2.32,1.28) circle (1pt);
\path[draw,line width=1.3pt] (2,0) -- (2.32,1.28);

\fill (2.32,1.6) circle (1pt);
\fill (2.31,1.54) circle (1pt);
\fill (2.30,1.48) circle (1pt);
\fill (2.29,1.43) circle (1pt);
\fill (2.28,1.38) circle (1pt);
\path[draw,line width=1.3pt] (2,0) -- (2.28,1.38);

\fill (5,2) circle(1pt);
\fill (4.67,2) circle(1pt);
\fill (4.5,2) circle(1pt);
\fill (4.4,2) circle(1pt);
\fill (4.33,2) circle(1pt);
\fill (4.29,2) circle(1pt);
\fill (4.25,2) circle(1pt);
\fill (4.22,2) circle(1pt);
\fill (4.2,2) circle(1pt);
\path[draw,line width=1.3pt] (4,2) -- (4.2,2);

\fill (3.5,1.5) circle(1pt);
\fill (3.67,1.67) circle(1pt);
\fill (3.71,1.71) circle(1pt);
\fill (3.75,1.75) circle(1pt);
\fill (3.78,1.78) circle(1pt);
\fill (3.8,1.8) circle(1pt);
\path[draw,line width=1.3pt] (3.8,1.8) -- (4,2);

\fill (3.45,0.73) circle (1pt);
\fill (3.33,0.67) circle (1pt);
\fill (3.23,0.62) circle (1pt);
\fill (3.14,0.57) circle (1pt);
\fill (3.07,0.53) circle (1pt);
\fill (3,0.5) circle (1pt);
\fill (2.94,0.47) circle (1pt);
\fill (2.89,0.44) circle (1pt);
\fill (2.84,0.42) circle (1pt);
\fill (2.8,0.4) circle (1pt);
\fill (2.76,0.38) circle (1pt);
\path[draw,line width=1.3pt] (2,0) -- (2.76,0.38);

\fill (2.36,0.76) circle (1pt);
\fill (2.35,0.73) circle (1pt);
\path[draw,line width=1.3pt] (2,0) -- (2.35,0.73);

\fill (2.52,0.77) circle (1pt);
\fill (2.5,0.75) circle (1pt);
\path[draw,line width=1.6pt] (2,0) -- (2.48,0.73);

\fill (2.31,0.78) circle (1pt);
\path[draw,line width=1.3pt] (2,0) -- (2.31,0.78);

\fill (2.216,1.3) circle (1pt);
\fill (2.21,1.26) circle (1pt);
\path[draw,line width=1.3pt] (2,0) -- (2.205,1.23);
\fill (2.186,1.3) circle (1pt);
\fill (2.18,1.27) circle (1pt);
\path[shade,draw] (2,0)--(2.22,1.33)--(2,1.33)-- cycle;

\fill (8,2) circle (1pt);
\fill (8.67,2.67) circle (1pt);
\fill (9,3) circle (1pt);
\fill (9.2,3.2) circle (1pt);

\fill (7.33,2.67) circle (1pt);

\fill (5.2,1.6) circle (1pt);

\fill (8.4,1.6) circle (1pt);
\fill (6.8,1.6) circle (1pt);

\fill (10,2.66) circle (2pt);

\fill (6,2) circle (1pt);
\fill (8,2) circle (1pt);

\fill (8.67,1.33) circle (1pt);

\draw[line width=0.5pt] (2,1.54) .. controls(6,2.48) .. (10,3.6);  

\path[draw,line width=0pt] (2.1,-1.3) -- (7.9,-1.3) node[pos=0.5,sloped,above] {Figure 5: $\Cliff(C) =2$, $2 \leq \mu \leq 3$}; 

\end{tikzpicture}

\begin{tikzpicture}
\path[draw][->] (0,-0.06) -- (0,8.1) node[pos=0.99,above] {$\lambda$};
\path[draw][->] (0,0 ) -- (10.2,0) node[pos=1.02,right] {$\mu$} node[pos=0,below] {$3$} node[pos=0.98,below] {$5$} node[pos=0,left] {$1$}; 
\draw[line width=0.5pt] (0,2.25) .. controls(5,3.9) .. (10,6.08);              
\path[draw] (10, 0) -- (10,8)  

node[pos=0.4,right] {}; 
\path[draw,line width=1.2pt] (0,2.5) -- (10,7.5) node[pos=0,left] {$3/2$} ; 
\path[fill=black] (5,0)--(10,0)--(10,5)--cycle;
\path[fill=black] (0,0)--(5,0)--(5,0.83)-- cycle;
\path[shade,draw] (5,0)--(10,5)--(5,3.33)-- cycle;
\path[draw] (5,0) -- (5,8)  node[pos=0,below]{$4$};
\path[draw][dashed] (0,5) -- (10,5) node[pos=0,left] {$2$}  ;
\path[draw][dashed] (0,7.5) -- (10,7.5) node[pos=0,left]{$5/2$};
\path[draw][dashed] (0,0.83) -- (5,0.83) node[pos=0,left]{$7/6$};

\path[shade,draw] (0,0)--(0.13,0.83)--(0,0.83)-- cycle;

\fill (1.67,3.33) circle (2pt);

\fill (7.5,5) circle (1.2pt);
\fill (6.67,5) circle (1.2pt);
\fill (6.25,5) circle (1.2pt);
\fill (6,5) circle (1.2pt);

\fill (2.5,2.5) circle (1.2pt);
\fill (3.33,3.36) circle (1.2pt);
\fill (3.75,3.75) circle (1.2pt);
\fill (4,4) circle (1.2pt);

\fill (0,1.67) circle (1.2pt);
\fill (1.67,1.67) circle (1.2pt);
\fill (3.33,1.67) circle (1.2pt);
\fill (5,1.67) circle (1.2pt);
\fill (2.5,2.5) circle (1.2pt);
\fill (1.25,2.5) circle (1.2pt);
\fill (0.83,2.5) circle (1.2pt);
\fill (0.63,2.5) circle (1.2pt);
\fill (0.5,2.5) circle (1.2pt);
\fill (5,2.5) circle (1.2pt);
\path[draw,dashed] (0,2.5) -- (5,2.5);

\fill (0,2.5) circle (2pt);
\fill (10,7.5) circle (2pt);
\fill (5,5) circle (2pt);
\fill (10,7.5) circle (2pt);

\fill (10,5) circle (1.2pt);
\fill (6.25,5) circle (1.2pt);
\fill (7.5,5) circle (1.2pt);
\fill (6,5) circle (1.2pt);
\fill (8.33,5) circle (1.2pt);
\fill (8.75,5) circle (1.2pt);
\fill (7,5) circle (1.2pt);
\fill (8,5) circle (1.2pt);
\fill (9,5) circle (1.2pt);

\fill (3.75,3.75) circle (1.2pt);
\fill (4,4) circle (1.2pt);

\fill (1,1) circle (1.2pt);
\fill (0.83,0.83) circle (1.2pt);
\fill (0.71,0.71) circle (1.2pt);
\fill (0.63,0.63) circle (1.2pt);
\fill (0.56,0.56) circle (1.2pt);
\fill (0.5,0.5) circle (1.2pt);
\fill (0.45,0.45) circle (1.2pt);
\fill (0.42,0.42) circle (1.2pt);  
\fill (0.38,0.38) circle (1.2pt); 
\fill (0.36,0.36) circle (1.2pt);
\fill (0.33,0.33) circle (1.2pt);
\path[draw,line width=1.4pt] (0,0) -- (0.33,0.33);

\fill (0.5,1.0) circle (1.2pt);
\fill (0.45,0.9) circle (1.2pt);
\fill (0.42,0.84) circle (1.2pt);
\fill (0.38,0.76) circle (1.2pt);
\fill (0.36,0.72) circle (1.2pt);
\fill (0.33,0.67) circle (1.2pt);
\fill (0.31,0.63) circle (1.2pt);
\fill (0.29,0.58) circle (1.2pt);  
\path[draw,line width=1.4pt] (0,0) -- (0.28,0.56);

\fill (0.33,1) circle (1.2pt);
\fill (0.31,0.94) circle (1.2pt);
\fill (0.29,0.88) circle (1.2pt);
\fill (0.28,0.83) circle (1.2pt);
\fill (0.26,0.79) circle (1.2pt); 
\path[draw,line width=1.4pt] (0,0) -- (0.26,0.79);

\fill (0.25,1) circle (1.2pt);
\fill (0.24,0.95) circle (1.2pt);
\fill (0.23,0.91) circle (1.2pt);
\fill (0.22,0.87) circle (1.2pt);
\path[draw,line width=1.4pt] (0,0) -- (0.22,0.87);

\fill (0.2,1) circle (1.2pt);
\fill (0.19,0.96) circle (1.2pt);
\fill (0.18,0.93) circle (1.2pt);
\fill (0.18,0.89) circle (1.2pt);
\path[draw,line width=1.6pt] (0,0) -- (0.18,0.89);

\fill (1,1) circle (1.2pt);
\fill (0.91,0.45) circle (1.2pt);
\fill (0.83,0.42) circle (1.2pt);
\fill (0.77,0.38) circle (1.2pt);
\fill (0.71,0.36) circle (1.2pt);
\fill (0.67,0.33) circle (1.2pt);
\fill (0.63,0.31) circle (1.2pt);
\path[draw,line width=1.6pt] (0,0) -- (0.59,0.29);

\fill (0.32,0.48) circle (1.2pt);
\fill (0.31,0.47) circle (1.2pt);
\path[draw,line width=1.6pt] (0,0) -- (0.30,0.45);

\fill (1,1) circle (1.2pt);
\fill (2,1) circle (1.2pt);
\fill (3,1) circle (1.2pt);
\fill (4,1) circle (1.2pt);

\fill (2,1) circle (1.2pt);
\path[draw,dashed] (0,1) -- (5,1)  node[pos=1,right]{$6/5$};;

\path[draw,dashed] (0,1.67) -- (5,1.67)  node[pos=0,left]{$4/3$};;

\fill (0,1.66) circle (2pt);

\fill (1.25,1.25) circle (1pt);
\fill (2.5,1.25) circle (1pt);
\fill (3.75,1.25) circle (1pt);

\fill (4.17,0.833) circle (1pt);

\path[draw,dashed] (0,1.25) -- (5,1.25) node[pos=0,left] {$5/4$};

\path[draw,line width=0pt] (2.1,-1.3) -- (7.9,-1.3) node[pos=0.5,sloped,above] {Figure 6: $\Cliff(C) =2$, $3 \leq \mu \leq 5$}; 

\end{tikzpicture}

\end{document}